\documentclass[a4paper,11pt]{amsart}

\pdfoutput=1

\usepackage[T1]{fontenc}
\usepackage[utf8]{inputenc}

\usepackage{amssymb,amsfonts,amsmath,amsthm}
\usepackage{url}
\usepackage{color}
\usepackage{enumerate}
\usepackage{epsfig,graphicx,psfrag}
\usepackage{tikz}
\usepackage{asymptote}
\usepackage{pinlabel}

\usepackage{mathabx}
\usepackage{MnSymbol}

\usepackage[top=1.2in,bottom=1.4in,left=1.4in,right=1.4in]{geometry}

\input{xy}
\xyoption{all}
\objectmargin={3mm}

\newcounter{notes}
\newcommand{\ignore}[1]{}

\definecolor{maxime}{rgb}{0.8,0,0.2}
\definecolor{fred}{rgb}{0,0,0.8}

\newtheorem{theorem}{Theorem}
\newtheorem{proposition}[theorem]{Proposition}
\newtheorem{corollary}[theorem]{Corollary}
\newtheorem{lemma}[theorem]{Lemma}

\newtheorem{fact}[theorem]{Fact}

\newtheorem{observation}[theorem]{Observation}

\theoremstyle{definition}

\newtheorem{remark}[theorem]{Remark}

\newtheorem{example}[theorem]{Example}

\newtheoremstyle{theoremwithref}{}{}{\itshape}{}{\bfseries}{.}{.5em}{#1 #2 #3}
\theoremstyle{theoremwithref}

\newcommand{\R}{\mathbb{R}}

\newcommand{\Z}{\mathbb{Z}}
\newcommand{\N}{\mathbb{N}}

\newcommand{\bbS}{\mathbb{S}}
\newcommand{\T}{\mathbb{T}}

\newcommand{\Aut}{\mathrm{Aut}}
\newcommand{\Point}{\mathrm{Point}}

\newcommand{\Diff}{\mathrm{Diff}}

\newcommand{\One}{\mathrm{One}}

\newcommand{\Mod}{\mathrm{Mod}}

\DeclareMathOperator{\Homeo}{Homeo}

\newcommand{\NCfin}[1]{\mathcal{NC}_{#1}^\dagger}

\newcommand{\TransFin}{\NCfin{\pitchfork}}
\newcommand{\TransFinLisse}{\mathcal{NC}_\pitchfork^{\dagger\infty}}

\newcommand{\Cfin}[1]{\mathcal{C}_{#1}^\dagger}
\newcommand{\Afin}[1]{\mathcal{A}_{#1}^\dagger}
\newcommand{\NAfin}[1]{\mathcal{NA}_{#1}^\dagger}
\newcommand{\point}{\mathrm{Point}}
\newcommand{\cA}{\mathcal{A}}

\newcommand{\adjacent}{\diamondvert}
\newcommand{\width}{\mathrm{Width}}

\title{Automorphisms of some variants of fine graphs}

\author{Fr\'ed\'eric Le Roux}
\address{Sorbonne Universit\'es, UPMC Univ.\ Paris 06,
Institut de Math\'ematiques de Jussieu-Paris Rive Gauche,
UMR 7586, CNRS, Univ. Paris Diderot, Sorbonne
Paris Cit\'e, 75005 Paris, France}
\email{frederic.le-roux@imj-prg.fr}

\author{Maxime Wolff}
\address{Sorbonne Universit\'es, UPMC Univ.\ Paris 06,
Institut de Math\'ematiques de Jussieu-Paris Rive Gauche,
UMR 7586, CNRS, Univ. Paris Diderot, Sorbonne
Paris Cit\'e, 75005 Paris, France}
\email{maxime.wolff@imj-prg.fr}

\begin{document}

\numberwithin{theorem}{section}

\begin{abstract}
  Recently Bowden, Hensel and Webb defined the {\em fine curve graph} for
  surfaces, extending the notion of curve graphs for the study of
  homeomorphism or diffeomorphism groups of surfaces.
  Later Long, Margalit, Pham, Verberne and Yao proved that for a closed
  surface of genus~$g\geqslant 2$,
  the automorphism group of the fine graph is naturally isomorphic
  to the homeomorphism group of the surface.
  We extend this result to the torus case $g=1$; in fact our method works for   
  more general surfaces, compact or not, orientable or not.
  We also discuss the case of a smooth version of the fine graph.
\vspace{0.2cm}

\end{abstract}

\maketitle

\sloppy
\section{Introduction}

\subsection{Context and results}

For a connected, compact surface $\Sigma_g$ of genus $g\geqslant 1$,
Bowden, Hensel and Webb~\cite{BHW} recently introduced the
{\em fine curve graph} $\mathcal{C}^\dagger(\Sigma)$,
as the graph whose vertices are all the essential closed curves on
$\Sigma$, with an edge between two vertices $a$ and $b$ whenever
$a\cap b=\emptyset$, if $g\geqslant 2$, and whenever
$|a\cap b|\leqslant 1$ if $g=1$. They proved that for every $g\geqslant 1$,
the graph $\mathcal{C}^\dagger(\Sigma)$ is hyperbolic, and derived a
construction of an infinite dimensional family of quasi-morphisms on
$\Homeo_0(\Sigma)$, thereby answering long standing questions of
Burago, Ivanov and Polterovich.

The ancestor of the fine graph is the
usual curve complex of a surface $\Sigma$, \textsl{i.e.}, the complex whose vertices
are the isotopy classes of essential curves, with an edge (or a simplex,
more generally) between some vertices if and only if they have disjoint
representatives. Since its introduction by Harvey~\cite{Harvey}, the curve
complex of a surface has been an extremely useful tool for the study of the
mapping class group $\Mod(\Sigma)$ of that surface, as it acts on it naturally.
In particular,
the fact that this complex is hyperbolic, discovered by Masur and Minsky,
has greatly improved the understanding of the mapping class groups (see~\cite{MM1,MM2}).
The result of Bowden, Hensel and Webb, promoting the
hyperbolicity of the curve complex to that of the fine curve graph, opens the
door both to the study of
what classical properties of usual curve complexes have counterparts in
the fine curve graph, and to the use of this graph to derive properties of
homeomorphism groups.
A first step in this direction was taken by
Bowden, Hensel, Mann, Militon and Webb~\cite{BHMMW}, who explored the metric
properties of the action of $\Homeo(\Sigma)$ on this hyperbolic graph.

A classical theorem by Ivanov~\cite{Ivanov} states that, when $\Sigma$ is
a closed surface of genus $g\geqslant 2$, the natural map
$\Mod(\Sigma)\to\Aut(\mathcal{C}(\Sigma))$ is an isomorphism.
Recently Long, Margalit, Pham, Verberne and Yao~\cite{LMPVY} proved
the following natural counterpart of Ivanov's theorem for fine graphs:
provided $\Sigma$ is a compact orientable surface of genus $g\geqslant 2$, the natural map
\[ \Homeo(\Sigma)\longrightarrow\Aut(\mathcal{C}^\dagger(\Sigma))  \]
is an isomorphism.
They also suggested that this map
(with the appropriate version of
$\mathcal{C}^\dagger$)
may also be an isomorphism when $g=1$, and
conjectured that the automorphism group of the fine curve graph
of smooth curves, should be nothing more than $\Diff(\Sigma)$.

In this article, we address both these questions.
Our motivation originates from the case of the torus: excited by~\cite{BHMMW},
 we wanted to understand
more closely the relation between the rotation set of homeomorphisms isotopic
to the identity and the metric properties of their actions on the fine graph.
This subject
will be treated in another article, joint with Passeggi and
Sambarino~\cite{Fantome}. The methods developed in the present article are valid not only for the torus but for a large class of surfaces.

We work on nonspherical surfaces
(\textsl{i.e.}, surfaces not embeddable in the $2$-sphere, or equivalently,
containing at least one nonseparating simple closed curve), orientable or not,
compact or not.
We consider the graph $\TransFin(\Sigma)$, whose vertices are
the nonseparating simple closed curves, and with an edge between two
vertices $a$ and $b$ whenever they are either disjoint, or have
exactly one, topologically transverse intersection point (see the beginning of Section~\ref{sec:Lemmes1-2-3} for more detail).
Our first result answers a problem raised in~\cite{LMPVY}.
\begin{theorem}\label{thm:AutC1}
  Let $\Sigma$ be a connected, nonspherical surface, without boundary.
  Then the natural map $\Homeo(\Sigma)\to\Aut(\TransFin(\Sigma))$
  is an isomorphism.
\end{theorem}
Our second result concerns the smooth version of fine graphs.
We consider the graph $\TransFinLisse(\Sigma)$ whose vertices are the
smooth nonseparating curves in $\Sigma$, with an edge between $a$
and $b$ if they
are disjoint or have one, transverse intersection point, in the differentiable
sense (in particular, $\TransFinLisse(\Sigma)$ is not the subgraph of
$\TransFin(\Sigma)$ induced by the vertices corresponding to smooth curves:
it has fewer edges). The following result partially confirms the
conjecture of~\cite{LMPVY}; here we restrict to the case of orientable surfaces for simplicity.
\begin{theorem}\label{thm:AutCFinLisse}
  Let $\Sigma$ be a connected, orientable, nonspherical surface, without boundary.
  Then all the automorphisms of $\TransFinLisse(\Sigma)$ are realized
  by homeomorphisms of~$\Sigma$.
\end{theorem}

In other words, if we denote by $\Homeo_{\infty \pitchfork}(\Sigma)$ the subgroup of
$\Homeo(\Sigma)$ preserving the collection of smooth curves and preserving
transversality, then the natural map
\[
\Homeo_{\infty \pitchfork}(\Sigma) \longrightarrow \Aut(\TransFinLisse(\Sigma))
\]
is an isomorphism. We were surprised to realize however that $\Homeo_{\infty \pitchfork}(\Sigma)$ is strictly larger than $\Diff(\Sigma)$.

\begin{proposition}\label{prop:PasDiff}
  Every surface $\Sigma$ admits a homeomorphism $f$
  such that
  $f$ and $f^{-1}$ preserve the set of smooth curves, and preserve
  transversality, but such that neither $f$ nor $f^{-1}$ is differentiable.
  In particular, the natural map 
  $$\Diff(\Sigma)\longrightarrow \Aut(\TransFinLisse(\Sigma))$$
  is not surjective.
\end{proposition}

\subsection{Idea of the proof of Theorem~\ref{thm:AutC1}}

The main step in this proof is the following.
\begin{proposition}\label{prop:Types}
  If $\{a,b\}$, or $\{a,b,c\}$ is a $2$-clique or
  a $3$-clique of $\TransFin(\Sigma)$ then, from the graph
  structure of $\TransFin(\Sigma)$, we can tell the type of the clique.
\end{proposition}
If $\{a_1,\ldots, a_n\}$ is an $n$-clique in the graph
$\TransFin(\Sigma)$, the homeomorphism type of the subset
$\cup_{j=1}^n a_j$ of $\Sigma$ will be called the {\em type} of the $n$-clique.
We will explore this only for $2$ and $3$-cliques. A $2$-clique
$\{a,b\}$, \textsl{i.e.}, an edge of the graph
$\TransFin(\Sigma)$,
may have two distinct types: the intersection $a\cap b$ may be empty
or not. For a $3$-clique $\{a,b,c\}$, up to permuting the curves $a$, $b$
and $c$, the cardinals of the intersections
$a\cap b$, $a\cap c$ and $b\cap c$, respectively, may be $(1,1,1)$,
or $(1,1,0)$, or $(1,0,0)$, or $(0,0,0)$. This determines the type of
the $3$-clique, except in the case $(1,1,1)$, where the intersection
points $a\cap b$, $a\cap c$ and $b\cap c$ may be pairwise distinct, in
which case we will speak of a $3$-clique of type {\em necklace}, or these
intersection points may be equal, in which case we will speak of a
$3$-clique of type {\em bouquet}, see
Figure~\ref{fig:Bouquet-Collier}.
\begin{figure}[htb]
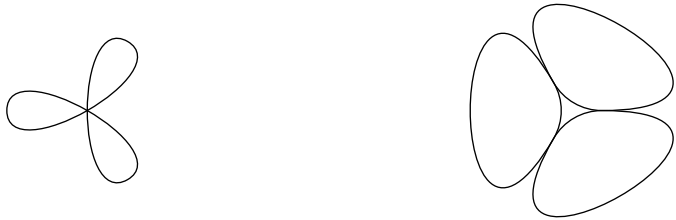

\begin{asy}
  import geometry;

  picture Bouquet, Collier;
  real L = 30, l = 12;
  path boucle = (0,0){dir(30)}..L*dir(60)..{dir(-90)}(0,0);
  draw (Bouquet, (0,0){dir(30)}..L*dir(60)..{dir(-90)}(0,0));
  draw (Bouquet, rotate(120)*boucle, (0,0));
  draw (Bouquet, rotate(-120)*boucle, (0,0));
  add(Bouquet, (0,0));

  path cercle = l*dir(0){dir(0)}..l*dir(120){dir(300)}..cycle;
  draw (Collier, cercle);
  draw (Collier, rotate(120)*cercle, (0,0));
  draw (Collier, rotate(-120)*cercle, (0,0));
  add(Collier, (180,0));
\end{asy}
\caption{A bouquet (left) and a necklace (right) of three circles.}
\label{fig:Bouquet-Collier}
\end{figure}

The main bulk of the proof of Proposition~\ref{prop:Types} consists in
distinguishing the $3$-cliques of type
necklace from any other $3$-clique of $\TransFin(\Sigma)$.
Here, the key is that among all the $3$-cliques, the cliques $\{a,b,c\}$ of
type necklace are exactly those such that the union $a\cup b\cup c$ contains
nonseparating simple closed curves other than $a$, $b$ and $c$.
In terms of the graph structure, this leads to the following property,
denoted by $N(a,b,c)$, which turns out to characterize these cliques:

{\em There exists
a finite set $F$ of at most $8$ vertices of $\TransFin(\Sigma)$,
all distinct from $a$, $b$ and $c$, such that every vertex $d$ connected
to $a$, $b$ and $c$ in this graph, is connected to at least one element
of $F$.}

From this,
we will easily characterize all the configurations of $2$-cliques and $3$-cliques
in terms of similar statements in the first order logic of
the graph $\TransFin(\Sigma)$.

Now, let $T(\TransFin(\Sigma))$ denote the set of edges $\{a,b\}$ of
$\TransFin(\Sigma)$ satisfying $|a\cap b|=1$. Then
we have a map
\[ \Point\colon \ T(\TransFin(\Sigma))\to\Sigma, \]
which to each edge $\{a,b\}$ of $\TransFin(\Sigma)$, associates the intersection point
$a\cap b$.
The next step in the the proof of Theorem~\ref{thm:AutC1} now consists in
characterizing the equality $\Point(a,b)=\Point(c,d)$ in terms of the
structure of the graph. This characterization shows that every automorphism of
$\TransFin(\Sigma)$ is realized by some bijection of $\Sigma$; then we prove that such a
bijection is necessarily a homeomorphism (see Proposition~\ref{prop:BijHomeo}).

In order to characterize the equality $\Point(a,b)=\Point(c,d)$,
we introduce on $T(\TransFin(\Sigma))$ the relation $\adjacent$
generated, essentially (see section~\ref{sec:Adjacency} for details),
by
$(a,b)\adjacent(b,c)$ if $(a,b,c)$ is a $3$-clique of type bouquet.
Obviously, if $(a,b)\adjacent(c,d)$ then $\Point(a,b)=\Point(c,d)$.
Interestingly, the converse is false, but we can still use this idea in
order to characterize the points of $\Sigma$ in terms of the graph structure
of $\TransFin(\Sigma)$.

This subtlety between the relation $\adjacent$ and the equality of points
is related to the non smoothness of the curves involved, and more precisely,
to the fact that a curve may spiral infinitely with respect to another curve in a
neighborhood of a common point.
We think that this phenomenon is of independent interest and we investigate it in Section~\ref{sec:LocalSubgraphs}. In particular, we can easily state, in terms of the graph structure
of $\TransFin(\Sigma)$, an
obstruction for a homeomorphism to be conjugate to a $C^1$-diffeomorphism,
see Section~\ref{ssec:Tourbillons}.

\subsection{Ideas of the proof of Theorem~\ref{thm:AutCFinLisse}}

In the smooth case, the adaptation of our proof of Theorem~\ref{thm:AutC1}
fails from the start: indeed, the closed curves contained in the union
$a\cup b\cup c$ of a necklace, and distinct from $a$, $b$ and $c$, are
not smooth.
This suggests the idea to use sequences of curves (at the expense of losing the
characterizations of configurations in terms of first order logic).

This time it is easiest to first characterize disjointness of curves (see Lemma~\ref{lem:DisjointLisse}), and then recover
the different types of $3$-cliques. Then the strategy follows the $C^0$ case.

Once we start to work with sequences, it is natural to say that a sequence
$(f_n)$
of curves not escaping to infinity
converges to $a$ in some weak sense, if for every vertex $d$ such
that $\{a,d\}$ is an edge of the graph, $\{f_n,d\}$ is also an edge for all
$n$ large enough.
As it turns out, this property implies convergence in $C^0$-sense to
$a$, and is implied by convergence in $C^1$-sense. But it is not equivalent
to the convergence in $C^1$-sense, and it is precisely this default of
$C^1$-convergence that enables us to distinguish between disjoint or
transverse pairs of curves.

Interestingly,
this simple criterion for disjointness has no counterpart in
the $C^0$-setting. Indeed, in that setting, no sequence of curves converges
in this weak sense:
given a curve $a$, and a sequence $(f_n)$ of curves with, say, some
accumulation point in $a$, we can build a curve $d$ intersecting $a$ once
transversally (topologically), but oscillating so much that it itersects every
$f_n$ several times.
From this perspective, none of our approaches in the
$C^0$-setting and in the $C^\infty$-setting are directly adaptable to the other.




\subsection{Further comments}

We can imagine many variants of fine graphs. For example, in the
arXiv version of~\cite{BHW}, for the case of the torus they worked with the
graph $\TransFin(\Sigma)$ on which we are working here, whereas in the
published version, they changed to
a fine graph in which two curves $a$ and $b$ are still related by an edge
when they have one intersection, not necessarily transverse.

More generally, in the spirit of Ivanov's metaconjecture,
we expect that the group of
automorphisms should not change from
any reasonable variant to another. And indeed,
using the ideas of \cite[Section~2]{LMPVY} and those presented here,
we can navigate between various versions of fine graphs, and recover,
from elementary properties of one version, the configurations defining
the edges in another version, thus proving that their automorphism groups
are naturally isomorphic.
From this perspective, it seems satisfying to
recover the group of homeomorphisms of the surface as the automorphism
group of any reasonable variant of the fine graph.
In this vein, we should mention that the results of~\cite[Section~2]{LMPVY}
directly yield a natural map
$\Aut(\Cfin{}(\Sigma))\to\Aut(\TransFin(\Sigma))$, and from there, our
proof of Theorem~\ref{thm:AutC1} may be used as
an alternative proof of their main result.


All reasonable variants of the fine graphs should be quasi-isometric,
and a unifying theorem (yet out of reach today, as it seems to us)
would certainly be a counterpart of the theorem by
Rafi and Schleimer~\cite{RafiSchleimer}, which states
that every quasi-isometry of the usual curve graph is bounded distance
from an isometry.

\subsection{Organization of the article}
Section~\ref{sec:Lemmes1-2-3} is devoted to the recognition of the
$3$-cliques in the $C^0$-setting, and of some other configurations regarding
the nonorientable case. We encourage the reader to skip, at first reading,
everything that concerns the nonorientable case: these points shoud be easily
identified, and this halves the length of the proof.
In Section~\ref{sec:AutHomeo} we prove Theorem~\ref{thm:AutC1}.
In Section~\ref{sec:LocalSubgraphs} we characterize, from the topological
viewpoint, the relation $\adjacent$ introduced above in terms of the graph
structure, and deduce our obstruction to differentiability.
Finally in Section~\ref{sec:AutCFinLisse} we prove Theorem~\ref{thm:AutCFinLisse}
and Proposition~\ref{prop:PasDiff}.

\subsection*{Acknowledgments}

We thank Kathryn Mann for encouraging discussions, and Dan Margalit for his extensive feedback on a preliminary version of this manuscript.

\section{Recognizing configurations of curves}\label{sec:Lemmes1-2-3}

\subsection{Standard facts and notation}
We will use, often without mention, the following easy or standard facts for curves
on surfaces.

The first is the classification of connected, topological surfaces with
boundary (not necessarily compact).
In particular, every topological surface admits a smooth structure.
Given a closed curve $a$ in a surface $\Sigma$, we can apply this classification
to $\Sigma\smallsetminus a$ and understand all possible configurations of simple
curves;
this is the so-called {\em change of coordinates
principle} in the vocabulary of the book of Farb and Margalit~\cite{FarbMargalit}.

In particular, every
closed curve $a$ has a neighborhood homeomorphic to an annulus or a M\"obius strip
in which $a$ is the ``central curve''.
Very often in this article, we will consider the curves $a'$ obtained by
deforming $a$ in such a neighborhood, so that $a'$ is disjoint from $a$ in the
first case, or intersects it once, transversely, in the second, as in
Figure~\ref{fig:FigureFacts}. We will say that $a'$ is obtained by {\em pushing
$a$ aside}.

The change of coordinates principle also
applies to
finite graphs embedded in $\Sigma$: there is a homeomorphism of $\Sigma$
that sends any given graph to
a smooth graph, such that all edges connected to a given vertex leave it in distinct directions.
In the simple case when the graph is the union of two or three simple closed curves that pairwise intersect at most once, this observation 
justifies the description of the possible configurations of cliques in
the introduction. This also enables, provided two curves $a$ and $b$ intersect
at a single point (or more generally at a finite number of points), to speak
of a {\em transverse} (also called {\em essential}), or to the contrary
{\em inessential}, intersection point, as we did in the introduction.


Here are two other useful facts.
\begin{fact}\label{fact:NonSep}
  A simple closed curve $a$ in a surface, is nonseparating if and only
  if there exists a closed curve $b$, such that $a\cap b$ is a
  single point and this intersection is transverse.
\end{fact}
\begin{fact}\label{fact:TroisArcsSep}
  Let $p,q$ be two distinct points, and $x,x',x''$ three simple arcs,
  each with end-points $p$ and $q$, such that
  \[  x \cap x' = x' \cap x'' = x \cap x'' = \{p,q\}.  \]
  If two of the three curves $x \cup x', x' \cup x'', x'' \cup x$ are
  separating, then the third one is also separating.
\end{fact}
\begin{proof}
  Denote $y=x\smallsetminus\{p,q\}$, the arc $x$ without its ends, and
  similarly, define $y'$ and $y''$. Suppose $x\cup x'$ and $x\cup x''$ are
  separating.
  Denote by $\Sigma_1, \Sigma_2$, resp. $\Sigma_3, \Sigma_4$,
  the components of
  $\Sigma\smallsetminus(x\cup x')$, resp. $\Sigma\smallsetminus(x\cup x'')$,
  where $\Sigma_2$ contains $y''$ and $\Sigma_4$ contains $y'$.
  By looking at neighborhoods of $p$ and $q$
  (see Figure~\ref{fig:FigureFacts}, left), we see that
  $\Sigma'=\Sigma_2\cap\Sigma_4$ is non empty, and that the arc $y$
  bounds $\Sigma_1$ on one side, and $\Sigma_3$ on the other, so
  $\Sigma''=\Sigma_1\cup y\cup \Sigma_3$ is a surface.
  Now, $\Sigma\smallsetminus(x'\cup x'')=\Sigma'\cup\Sigma''$, and
  $\Sigma'$ and $\Sigma''$ are disjoint by construction.
\end{proof}
\begin{figure}[htb]
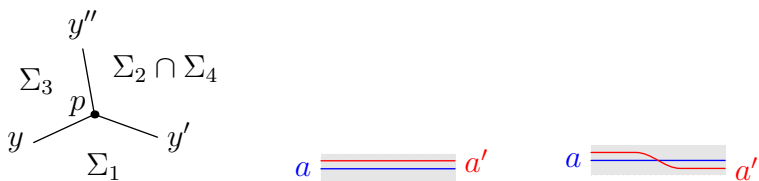

\hfill
\begin{asy}
  import geometry;
  
  real R = 25;

  dot ((0,0)); label("$p$", (0,0), WNW);
  draw ((0,0)--R*dir(100)); draw(R*dir(205)--(0,0)--R*dir(-20));
  label("$y''$", R*dir(100), N);
  label("$y$", R*dir(205), W);
  label("$y'$", R*dir(-20), E);

  label("$\Sigma_1$", 0.8*R*dir(-80));
  label("$\Sigma_3$", R*dir(150));
  label("$\Sigma_2\cap\Sigma_4$", 0.7*R*dir(80), E);
\end{asy}
\hfill
\begin{asy}
  import geometry;
  
  real L = 50, h = 11, e = 3;
  path contour = (0,0)--(L, 0)--(L, h)--(0, h)--cycle;
  fill (contour, lightgray);  draw(contour, lightgray+dotted);
  draw ((0,h/2)--(L,h/2), blue);  label("$a$", (0, h/2), W, blue);
  draw ((0, h/2+e)--(L, h/2+e), red); label("$a'$", (L, h/2+e), E, red);
\end{asy}
\hfill
\begin{asy}
  import geometry;
  
  real L = 50, h = 11, e = 3;
  path contour = (0,0)--(L, 0)--(L, h)--(0, h)--cycle;
  fill (contour, lightgray);  draw(contour, lightgray+dotted);
  draw ((0,h/2)--(L,h/2), blue);  label("$a$", (0, h/2), W, blue);
  draw ((0, h/2+e)..(1/3*L, h/2+e){right}..(2/3*L, h/2-e){right}..(L, h/2-e), red);
  label("$a'$", (L, h/2-e), E, red);
\end{asy}
\hspace{3cm}
\caption{Left: a neighborhood of $p$. Center: pushing a two-sided curve $a$.
Right: pushing a one-sided curve $a$.}
\label{fig:FigureFacts}
\end{figure}

\subsection{Properties characterizing geometric configurations}
Now we list the properties, in terms of the graph
$\TransFin(\Sigma)$, that will be used as
characterizations of certain configurations of curves.
This allows us to specify the statement of Proposition~\ref{prop:Types}, which
will be proved in the next paragraph, and define the relation
$\adjacent$ in terms of the graph $\TransFin(\Sigma)$.

In the following, the letters $N, D, T, B$ respectively stand for
necklace, disjoint, transverse, and bouquet.
If $a,b,c$ are vertices of this graph, we will denote by:
\begin{itemize}
\item[$\bullet$] $N(a,b,c)$ the property that $\{a,b,c\}$ is a $3$-clique
  of $\TransFin(\Sigma)$ and there exists a finite set $F$ of at most $8$
  vertices of $\TransFin(\Sigma)$, all distinct from $a$, $b$ and $c$,
  such that for every vertex $d$ such that $\{a,b,c,d\}$ is a $4$-clique,
  there is an edge from $d$ to at least one element of $F$,
\item[$\bullet$] $D(a,b)$ the property that $\{a,b\}$ is an edge of
  $\TransFin(\Sigma)$ and there does not
  exist a vertex $d$ such that $N(a,b,d)$ holds,
\item[$\bullet$] $T(a,b)$ the property that $\{a,b\}$ is an edge of
  $\TransFin(\Sigma)$ and $D(a,b)$ does not hold,
\item[$\bullet$] $B(a,b,c)$ the property that $T(a,b)$, $T(a,c)$, $T(b,c)$
  all hold but $N(a,b,c)$ does not.
\end{itemize}

The following proposition is the main part of
Proposition~\ref{prop:Types}.
\begin{proposition}\label{prop:TypesPrecis}
  Let $a$, $b$ and $c$ be vertices of the graph $\TransFin(\Sigma)$.
  Property $N(a,b,c)$ holds if and only if $\{a,b,c\}$
  is a $3$-clique of type necklace of $\TransFin(\Sigma)$.
\end{proposition}

The proof of Proposition~\ref{prop:TypesPrecis} will occupy the next
paragraph. The following corollary complements
Proposition~\ref{prop:TypesPrecis} and provides a precise version of
Proposition~\ref{prop:Types}.
\begin{corollary}\label{coro:TypesPrecis}~

  \begin{itemize}
 	\item Property $D(a,b)$ holds if and only if the curves $a$ and
      $b$ are disjoint.
	\item Property $T(a,b)$ holds if and only if $a$ and $b$ have
      a unique intersection point and the intersection is transverse.
    \item Property $B(a,b,c)$ holds if and only if $\{a, b, c\}$
      is a $3$-clique of type bouquet.
  \end{itemize}
\end{corollary}

\begin{proof}
  Let $a$ and $b$ be neighbors in the graph $\TransFin(\Sigma)$.
  Of course, if $a$ and $b$ are disjoint, then $D(a,b)$ holds: there does
  not exist a curve $d$ such that $N(a,b,d)$, since this would mean
  that $\{a,b,d\}$ is of type necklace and then, by definition, $a$ and $b$
  would intersect.
  Conversely, suppose that $a$ and $b$ are not disjoint, and let us prove
  that $D(a,b)$ does not hold, \textsl{i.e.}, let us find a curve $c$
  such that $\{a,b,c\}$ is a $3$-clique of type necklace.
  In the case when one of $a$ or $b$ is one-sided, up to exchanging the two,
  suppose $a$ is one-sided. Then we may push $a$ in order to find a curve $c$
  which makes a $3$-clique of type necklace with $a$ and $b$, see
  Figure~\ref{fig:OnComplete} (left). In the case when both $a$ and $b$ are
  two-sided, then by the change of coordinates principle,
  a regular neighborhood of $a\cup b$ is homeomorphic to a
  one-holed torus, embedded in $\Sigma$, with a
  choice of meridian and longitude
  coming from $a$ and $b$.
  In this torus, a curve $c$ with
  slope $1$ will form a $3$-clique of type necklace with $a$ and $b$,
  see Figure~\ref{fig:OnComplete}, right.
  \begin{figure}[htb]
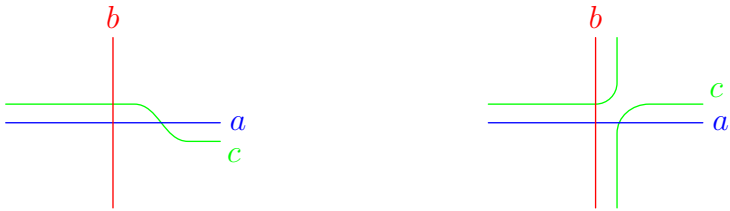

  \begin{asy}
  import geometry;
  
  picture commun, gauche, droite;

  path a = (-40, 0)--(40, 0), b = (0, -32)--(0, 32);
  path c1 = (-40, 7)--(8,7){right}..{right}(28,-7)--(40,-7);
  path c21 = (-40, 7)--(0, 7){right}..{up}(8, 15)--(8, 32);
  path c22 = (8, -32)--(8, -4){up}..{right}(20,7)--(40,7);

  draw(commun, a, blue); label(commun, "$a$", (40,0),E, blue);
  draw(commun, b, red); label(commun, "$b$", (0,32), N, red);
  draw(gauche, c1, green); label(gauche,"$c$", (40,-7),SE, green);
  draw(droite, c21, green); draw(droite, c22, green);
  label(droite,"$c$", (40,7),NE, green);
  
  add(gauche, (-90,0)); add(droite, (90,0));
  add(commun, (-90,0)); add(commun, (90,0));
  \end{asy}
  \caption{Completing $(a,b)$ to a $3$-clique of type necklace}
  \label{fig:OnComplete}
  \end{figure}
  This proves the first point.

  The second point is a straightforward
  consequence of the first, and the third simply follows from the second
  point together with Proposition~\ref{prop:TypesPrecis}.
 \end{proof}

\subsection{Proof of Proposition~\ref{prop:TypesPrecis}}

The following lemma is a key step in the proof of the direct implication in Proposition~\ref{prop:TypesPrecis}.
\begin{lemma}\label{lem:drenccc} 
  Let $\{a,b,c\}$ be a $3$-clique of $\TransFin(\Sigma)$ which is not
  of type necklace. Then there exists a vertex $d$ of
  $\TransFin(\Sigma)$ such that $\{a,b,c,d\}$ is a $4$-clique of
  $\TransFin(\Sigma)$, and such that $d$ meets every connected
  component of $\Sigma\smallsetminus(a\cup b\cup c)$.
\end{lemma}
Before entering the proof, we note that we cannot remove the hypothesis that
$\{a, b, c\}$ is not of type necklace. Indeed, in the flat torus
$\Sigma = \mathbb{R}^2/\mathbb{Z}^2$, consider three closed geodesics
$a,b,c$ respectively directed by $(1,0)$, $(0,1)$, and $(1, 1)$.
By pushing $c$ aside if necessary, we obtain a $3$-clique of type necklace.
The complement of $a \cup b \cup c$ in $\Sigma$ has three connected components,
and there is no curve $d$ satisfying the conclusion of the lemma.
\begin{proof}
  In all the proof, we will denote
  $\Sigma'=\Sigma\smallsetminus(a\cup b\cup c)$.
  Up to permuting the curves $a$, $b$ and $c$, we may suppose that the triple
  of cardinals of intersections, $(|a\cap b|, |a\cap c|, |b\cap c|)$, equals
  $(1,1,1)$, or $(1,1,0)$, or $(1,0,0)$, or~$(0,0,0)$. We will deal with
  these cases separately.
  
  Let us begin with the case $(0, 0, 0)$. If $\Sigma'$ is connected,
  then any curve $d$ making a $4$-clique with $(a,b,c)$ satisfies the Lemma.
  Such a curve can be found, for example, by pushing $a$ aside.
  If $\Sigma'$ has two connected components, denote them by $\Sigma_1$ and
  $\Sigma_2$. Since $a$, $b$ and $c$ are each nonseparating, at least two of
  the curves $a, b, c$ (say, $a$ and $b$) correspond to boundary components of
  both $\Sigma_1'$ and $\Sigma_2'$.
  Choose one point $x_a$ in $a$ and one point $x_b$ in $b$. For $i=1,2$, there
  is an arc $\gamma_i$ connecting $x_a$ to $x_b$ in $\Sigma_i$, and disjoint
  from the boundary of $\Sigma_i$ except at its ends. Then the
  curve $d=\gamma_1\cup\gamma_2$ satisfies the lemma
  (see Figure~\ref{fig:(0,0,0)}, left).
  It may also happen that $\Sigma'$ has three connected components, in which
  case we find a curve $d$ exactly in the same way, see
  Figure~\ref{fig:(0,0,0)}, right.
  \begin{figure}[htb]
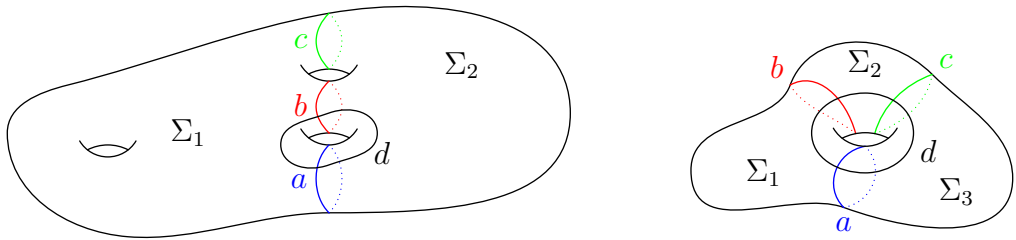

  \begin{asy}
  import geometry;
  
  real r = 1.5;

  picture trou, gauche, droite;
  path batrou = (-7,4)..(0,0)..(7,4);
  path hautrou = point(batrou,0.35)..(0,3)..point(batrou,1.65);
  draw (trou, batrou); draw(trou, hautrou);
  add(gauche, scale(r)*trou, r*(0,17)); add(gauche, scale(r)*trou, r*(0, 33));
  add(gauche, scale(r)*trou, r*(-55, 14));
  path contourg = (0,0){left}..(-50,-6)..(-80,20){up}..(-65, 32)..
  (0,50){dir(10)}..(60, 25){down}..cycle;
  point A = (-8, -28), B = (-28,18), C = (25, 22);
  path contourd = A{dir(160)}..(-65, -14){up}..B{dir(70)}..C{dir(-45)}..(55,-17){down}..cycle;
  
  path ga1 = (0,0){dir(135)}..(-3.2, 8)..(0,17){dir(45)};
  path ga2 = (0,0){dir(45)}..(3.2, 8)..(0,17){dir(135)};
  path gb1 = (0,20){dir(135)}..(-3.2, 27)..(0,33){dir(45)};
  path gb2 = (0,20){dir(45)}..(3.2, 27)..(0,33){dir(135)};
  path gc1 = (0,36){dir(135)}..(-3.2,43)..(0,50){dir(45)};
  path gc2 = (0,36){dir(45)}..(3.2,43)..(0,50){dir(135)};
  
  draw (gauche, scale(r)*contourg);
  draw (droite, contourd);
  draw (gauche, scale(r)*ga1, blue); draw(gauche, scale(r)*ga2, dotted+blue);
  label(gauche, "$a$", r*(-3,8), W, blue);
  draw (gauche, scale(r)*gb1, red); draw(gauche, scale(r)*gb2, dotted+red);
  label(gauche, "$b$", r*(-3,27), W, red);
  draw (gauche, scale(r)*gc1, green); draw(gauche, scale(r)*gc2, dotted+green);
  label(gauche, "$c$", r*(-3,43), W, green);
  label(gauche, "$\Sigma_1$", r*(-35, 20));
  label(gauche, "$\Sigma_2$", r*(33,37));
  path gd = (0,25){dir(200)}..(-12,17)..(0,12){dir(20)}..(12,20)..cycle;
  draw(gauche, scale(r)*gd);
  label(gauche, "$d$", r*(12,20), dir(-75));

  real R = 1.7;
  path da1 = A{dir(135)}..(0,-5){dir(10)};
  path da2 = A{dir(10)}..(0,-5){dir(135)};
  point PB = (0,-5)+R*point(hautrou, 0.6);
  point PC = (0,-5)+R*point(hautrou, 1.4);
  path db1 = B{dir(30)}..PB{dir(-75)};
  path db2 = B{dir(-60)}..PB{dir(-40)};
  path dc1 = PC{dir(75)}..C{dir(20)};
  path dc2 = PC{dir(20)}..C{dir(80)};
  path dd = (0,-15){right}..(18,0){up}..(0,15){left}..(-20,0){down}..cycle;
  draw (droite, da1, blue); draw(droite, da2, dotted+blue);
  draw (droite, db1, red); draw(droite, db2, dotted+red);
  draw (droite, dc1, green); draw(droite, dc2, dotted+green);
  draw (droite, dd);
  label (droite, "$a$", A, S, blue); label(droite, "$b$", B, NW, red);
  label (droite, "$c$", C, NE, green); label(droite, "$d$", (18,0), SE);
  label (droite, "$\Sigma_1$", (-38, -15)); label(droite, "$\Sigma_2$", (0, 26));
  label (droite, "$\Sigma_3$", (34, -22));
  
  add (droite, scale(R)*trou, (0,-5));
  add(gauche, (-80,0)); add(droite, (120,30));
  \end{asy}
  \caption{Finding $d$ in the case $(0,0,0)$}
  \label{fig:(0,0,0)}
  \end{figure}
  
  Next we deal with the case of intersections $(1, 0, 0)$. In this case,
  $a$ and $b$ intersect transversely, once, and $c$ is disjoint from $a\cup b$.
  By hypothesis, the curve $c$ is (globally) nonseparating.
  Consider the union $a\cup b$. If $a$ or $b$ is two-sided, then
  $a\cup b$ does not disconnect its regular neighborhoods.
  This is seen by travelling
  along a small band on one side of $a\cup b$,
  (see Figure~\ref{fig:DeuxCourbes}, left).
  In this case, $\Sigma'$ cannot have more connected components than
  $\Sigma\smallsetminus c$, hence $\Sigma'$ is connected, and any
  curve $d$ obtained by pushing $c$, as in the preceding case, satisfies
  the lemma. If both $a$ and $b$ are one-sided, then $a\cup b$ is locally
  disconnecting, so $\Sigma'$ may have up to two connected components.
  In this case, a curve $d$ obtained by pushing
  $a$ satisfies the lemma (See Figure~\ref{fig:DeuxCourbes}, right).
  \begin{figure}[htb]
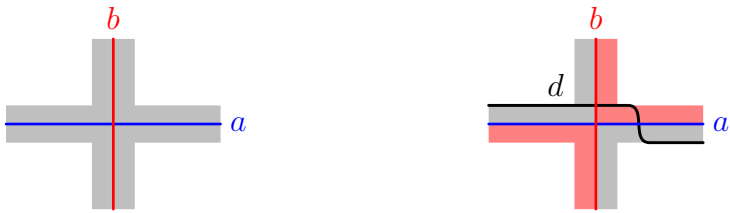

  \begin{asy}
  import geometry;
  picture commun, gauche, droite;

  path a = (-40, 0)--(40, 0), b = (0, -32)--(0, 32);
  path region1 = (-40, 0)--(0,0)--(0,32)--(-8,32)--(-8,7)--(-40,7)--cycle;
  path region2 = reflect((0,0),(0,1))*region1;
  path region3 = rotate(180)*region1, region4 = rotate(180)*region2;
  
  path d = (-40, 7)--(12,7){right}..{right}(20,-7)--(40,-7);

  fill(gauche, region1, mediumgray); fill(gauche, region3, mediumgray);
  fill(droite, region1, mediumgray); fill(droite, region3, mediumgray);
  fill(gauche, region2, mediumgray); fill(droite, region2, lightred);
  fill(gauche, region4, mediumgray); fill(droite, region4, lightred);
  draw(commun, a, 1pt+blue); label(commun, "$a$", (40,0),E, blue);
  draw(commun, b, 1pt+red); label(commun, "$b$", (0,32), N, red);
  draw(droite, d, 1.1pt+black); label(droite,"$d$", (-15,7),N);
  
  add(gauche, (-90,0)); add(droite, (90,0));
  add(commun, (-90,0)); add(commun, (90,0));
  \end{asy}
  \caption{Finding $d$ in case $(1,0,0)$}
  \label{fig:DeuxCourbes}
  \end{figure}
  
  Now assume we are in the case $(1, 1, 1)$ or $(1, 1, 0)$.
  Since $\{a, b, c\}$ is not of type necklace, note that  in any
  case $b$ and $c$ do not meet outside $a$. We first treat the sub-case
  when $a$ is two-sided. For this we consider any curve $d$ obtained by
  pushing $a$ aside, and we claim that $d$ meets every connected component
  of $\Sigma'$. Indeed, let $C$ be such a component.
  Of course the closure of $C$ meets $a$ or $b$ or $c$.
  Since both $b$ and $c$ meet $a$, it actually has to meet $a$, as we can see
  by traveling along $b$ or $c$ in $C$. More precisely, by following $b$
  or $c$ in both directions, we see that $C$ meets any neighborhood of $a$
  from both sides. Thus it meets $d$.

  It remains to treat the sub-case when $a$ is one-sided, first for the
  $(1,1,1)$ case, and then for the $(1,1,0)$ case. In the $(1,1,1)$ case, the
  curves $a$, $b$, $c$ play symmetric roles, and by the above argument it
  just remains to consider the case when they are all one-sided.
  Then the situation is depicted on Figure~\ref{fig:(1,1,1)Et(1,1,0)}, left:
  $a \cup b \cup c$ disconnects its regular neighborhoods into three connected
  components, and the figure shows a curve $d$, obtained by pushing $a$ aside,
  which intersects all three components and such that
  $\{a, b, c, d \}$ is a 4-clique.
  In the remaining case the curves
  $b$ and $c$ play symmetric roles,
  and there are three different cases to
  consider, regarding whether $b$ and $c$ are one or two-sided.
  These three cases are pictured in Figure~\ref{fig:(1,1,1)Et(1,1,0)},
  and in each case, we obtain $d$ by pushing $a$ aside.
  \begin{figure}[htb]
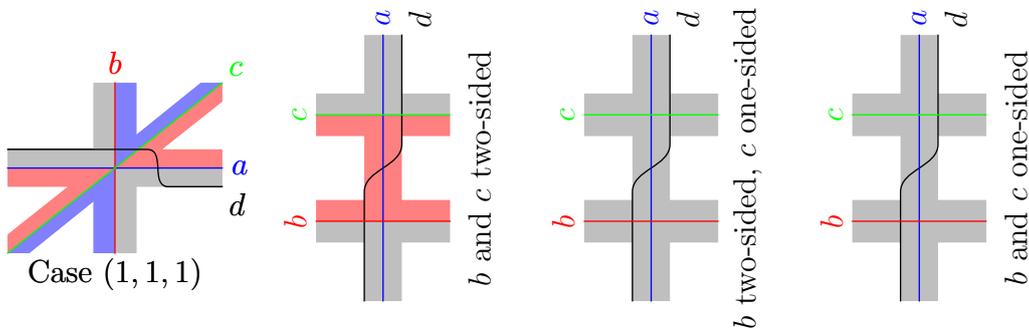

  \begin{asy}
  import geometry;
  
  picture gauche;
  path a = (-40, 0)--(40, 0), b = (0, -32)--(0, 32), c = (-40,-32)--(40,32);
  path region1 = (-40, 0)--(0,0)--(0,32)--(-8,32)--(-8,7)--(-40,7)--cycle;
  path region2 = (-40,0)--(-40,-7)--(-17.5,-7)--(-40,-25)--(-40,-32)--(0,0)--cycle;
  path region3 = (-40,-32)--(-32,-32)--(-8,-12.5)--(-8,-32)--(0,-32)--(0,0)--cycle;
  path region4 = rotate(180)*region1;
  path region5 = rotate(180)*region2, region6 = rotate(180)*region3;
  
  path d1 = (-40,7)--(40,7);
  path d2 = (-40, 7)--(12,7){right}..{right}(20,-7)--(40,-7);
  
  fill(gauche, region1, mediumgray); fill(gauche, region2, lightred);
  fill(gauche, region3, lightblue); fill(gauche, region4, mediumgray);
  fill(gauche, region5, lightred); fill(gauche, region6, lightblue);
  draw(gauche, a, blue); label(gauche, "$a$", (40,0),E, blue);
  draw(gauche, b, red); label(gauche, "$b$", (0,32), N, red);
  draw(gauche, c, green); label(gauche, "$c$", (40,32), NE, green);
  
  draw(gauche, d2); label(gauche,"$d$", (40,-7),SE);
  label(gauche, "Case $(1,1,1)$", (0,-40));
  
  add(gauche, (-150,0));
  add(gauche, (-150,0));
  
  //
  
  picture commun, un, deux, trois, quatre, cinq, six;
  path a = (-50,0)--(50,0), b = (-20,-25)--(-20,25), c = (20, -25)--(20,25);
  path d1 = (-50,7)--(50,7);
  path d2 = (-50, 7)--(-9,7){right}..{right}(9,-7)--(50,-7);
  path region1 = (-50,0)--(-20,0)--(-20,25)--(-28,25)--(-28,7)--(-50,7)--cycle;
  path region2 = reflect((0,0),(1,0))*region1;
  path region3 = (-20,-25)--(-12,-25)--(-12,-7)--(12,-7)--(12,-25)--(20,-25)--(20,0)--(-20,0)--cycle;
  path region4 = rotate(180)*region1, region5 = rotate(180)*region2;
  path region6 = rotate(180)*region3;
  
  //
  //
  //
  //
  //
  //
  //
  //
  //
  fill(quatre, region1, mediumgray); fill(quatre, region2, mediumgray);
  fill(quatre, region3, lightred); fill(quatre, region4, mediumgray);
  fill(quatre, region5, mediumgray); fill(quatre, region6, lightred);
  fill(cinq, region1, mediumgray); fill(cinq, region2, mediumgray);
  fill(cinq, region3, mediumgray); fill(cinq, region4, mediumgray);
  fill(cinq, region5, mediumgray); fill(cinq, region6, mediumgray);
  fill(six, region1, mediumgray); fill(six, region2, mediumgray);
  fill(six, region3, mediumgray); fill(six, region4, mediumgray);
  fill(six, region5, mediumgray); fill(six, region6, mediumgray);

  draw(commun, a, blue); label(commun, "$a$", (50,0), E, blue);
  draw(commun, b, red); label(commun, "$b$", (-20,25), N, red);
  draw(commun, c, green); label(commun, "$c$", (20,25), N, green);
  //
  //
  //
  draw(quatre, d2); label(quatre, "$d$", (50,-7), SE);
  draw(cinq, d2); label(cinq, "$d$", (50,-7), SE);
  draw(six, d2); label(six, "$d$", (50,-7), SE);
  //
  //
  //
  label(quatre, "$b$ and $c$ two-sided", (0,-36));
  label(cinq, "$b$ two-sided, $c$ one-sided", (0,-36));
  label(six, "$b$ and $c$ one-sided", (0,-36));
  //
  add(rotate(90)*quatre, (-50, 0)); add(rotate(90)*cinq, (50, 0)); add(rotate(90)*six, (150, 0));
  //
  add(rotate(90)*commun, (-50, 0)); add(rotate(90)*commun, (50, 0)); add(rotate(90)*commun, (150, 0));
  \end{asy}
  \caption{Finding $d$ in cases $(1,1,1)$ and $(1,1,0)$}
  \label{fig:(1,1,1)Et(1,1,0)}
  \end{figure}
\end{proof}
We deduce the following.
\begin{lemma}\label{lem:CestLeBouquet}
  Let $\{a,b,c\}$ be a $3$-clique of $\TransFin(\Sigma)$, not of type
  necklace. Let $(\alpha_1,\ldots,\alpha_n)$ be a finite family
  of vertices of $\TransFin(\Sigma)$, all distinct from $a$,
  $b$ and~$c$. Then there exists a vertex $d$ of $\TransFin(\Sigma)$, such that
  \begin{itemize}
  \item[$\bullet$] $\{a,b,c,d\}$ is a $4$-clique of $\TransFin(\Sigma)$,
  \item[$\bullet$] for all $j\in\{1,\ldots,n\}$, the intersection
    $d\cap \alpha_j$ is infinite; in particular, $\{d, \alpha_{j}\}$ is not an
    edge of $\TransFin(\Sigma)$.
  \end{itemize}
\end{lemma}
As a corollary, we get the direct implication in Proposition~\ref{prop:TypesPrecis}.
\begin{corollary}
  If $\{a,b,c\}$ is a $3$-clique not of type necklace, then $N(a,b,c)$ does not
  hold.
\end{corollary}
%
\begin{proof}[Proof of Lemma~\ref{lem:CestLeBouquet}]
  The hypotheses that $\{a,b,c\}$ is not of type necklace, and
  $\alpha_j\not\in\{a,b,c\}$, impose that for every $j$,
  the curve $\alpha_j$ is not contained in the union $a\cup b\cup c$.
  Hence, there exists a small subarc $\beta_j\subset\alpha_j$
  lying in the complement of $a\cup b\cup c$, and we may
  further suppose that these $n$ arcs are pairwise disjoint,
  and choose a point $x_j$ in $\beta_j$ for each $j$.

  Now, let $d_0$ be a vertex of $\NCfin{\leqslant 1}(\Sigma)$ as from
  Lemma~\ref{lem:drenccc}. Since $d_0$ meets every component of
  $\Sigma'= \Sigma\smallsetminus(a\cup b\cup c)$,
  we may perform a surgery on $d_0$, far from $a\cup b\cup c$,
  to obtain a new curve $d_1$ such that $\{a,b,c,d_1\}$ is still a
  $4$-clique, and $d_1$ still meets every component of
  $\Sigma'$, and $d_1$ passes through $x_1$. We may iterate this
  process, to get a curve $d_n$ which passes through $x_j$
  for every $j$, and such that $\{a,b,c,d_n\}$ is a $4$-clique.
  Finally, we may perform a last surgery on $d_n$, in the
  neighborhood of all $x_j$, in order to obtain a curve $d$
  such that for each $j$, $|\beta_j\cap d|$ is
  infinite.
\end{proof}

It remains to prove the converse implication in
Proposition~\ref{prop:TypesPrecis}, which we restate as
Lemma~\ref{lem:CestLeCollier}.
\begin{lemma}\label{lem:CestLeCollier}
  Let $\{a,b,c\}$ be a $3$-clique of $\TransFin(\Sigma)$ of type
  necklace. Then there exists a finite set $F$, of at most 8 vertices of
  $\TransFin(\Sigma)$ all distinct from $a$, $b$ and $c$, and
  such that every $d$ such that $\{a,b,c,d\}$ is a $4$-clique of
  $\TransFin(\Sigma)$ is connected by an edge to some element
  of~$F$.
\end{lemma}
In fact this set $F$ can be chosen explicitely, with cardinal at most~8,
as follows. If $\{a,b,c\}$ is a $3$-clique of type necklace, then there exists
a family of arcs $(x, X, y, Y, z, Z)$, all embedded in $\Sigma$, such that
$a=x\cup X$, $b=y\cup Y$, $c=z\cup Z$, and such that these
six arcs pairwise intersect at most at their ends. The union
$a\cup b\cup c$ may be viewed as a graph embedded in $\Sigma$,
and these six arcs are the edges of this embedded graph.
We let $F$ be the set of nonseparating curves, among the~8
curves $X\cup Y\cup Z$, $X\cup Y\cup z$,
$X\cup y\cup Z$ \textsl{etc.}, (there is one choice of
upper/lower case for each letter).
In the course of the proof of Lemma~\ref{lem:CestLeCollier}, we
will see that $F$ is nonempty, and satisfies the Lemma.
\begin{proof}[Proof of Lemma~\ref{lem:CestLeCollier}]
  Let $\{a,b,c\}$ be a $3$-clique of type necklace.
  Let $F$ be the set of nonseparating curves, as above, among all the~8
  curves $x\cup y\cup z$, $X\cup y\cup z$, $X\cup Y\cup z$, etc.
  Let $d$ be such that $\{a,b,c,d\}$ is a $4$-clique.
  Up to permuting the curves $a$, $b$ and $c$, we may suppose that
  $(|a\cap d|, |b\cap d|, |c\cap d|)$ equals
  $(1,1,1)$, or $(1,1,0)$, or $(1,0,0)$ or $(0,0,0)$~; our proof proceeds
  case by case.
  
  The easiest case is $(1, 0, 0)$. In this case,
  up to exchanging the arcs $X$ and $x$, we may suppose that
  $d$ intersects $a$ at an interior point of $X$, and is disjoint from
  all the other arcs. Consider $f=X\cup y\cup z$.
  This curve intersects
  $d$ at a unique point, transversely. It follows that $f$ is nonseparating.
  Hence $f\in F$ and $f$ satisfies the conclusion of the lemma.
  
  Now let us deal simultaneously with the cases $(1,1,1)$ and $(1,1,0)$.
  Suppose first that the intersections of
  $d$ with $a\cup b\cup c$ do not occur at the intersection points
  $a\cap b$, $a\cap c$ or $b\cap c$. Up to exchanging $x$ with $X$,
  $y$ with $Y$ and $z$ with $Z$, we may suppose that the intersections
  occur in the interior of the arcs $X$, $Y$, and $Z$ in the case $(1,1,1)$,
  and in the interior of the arcs $X$ and $Y$ in the case $(1,1,0)$.
  Now the curve
  $f=X\cup y\cup z$, for instance, satisfies the conclusion of the lemma.
  
  Now suppose that $d$ contains one of the points $a\cap b$, $a\cap c$ or $b\cap c$.
  In case $(1,1,1)$ we may suppose, up to permuting $a$, $b$ and $c$, that $d$
  contains the point $a\cap b$, and in the case $(1,1,0)$, this is automatic,
  as $d$ is disjoint from $c$.
  Now in any case, $d$ cannot
  contain $a\cap c$ nor $b\cap c$, because it intersects $a$ and $b$ only once.
  Hence, up to exchanging $z$ with $Z$, we may suppose $d\cap z=\emptyset$.
  In the neighborhood of the point $a\cap b$, up to homeomorphism, the configuration of
  our curves is as depicted in
  Figure~\ref{fig:dTransverse}, because all the intersections are supposed
  to be transverse. Then, up to exchanging $X$ with $x$ or $Y$ with $y$,
  we can suppose that the arc $X\cup Y$ has a transverse intersection with
  $d$, and then the arc $f = X\cup Y \cup z$ satisfies the conclusion of the lemma.
  \begin{figure}[htb]
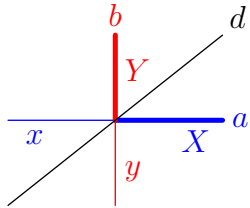

  \begin{asy}
  import geometry;
  
  path a1 = (-40, 0)--(0, 0), a2 = (0,0)--(40,0);
  path b1 = (0, -32)--(0, 0), b2 = (0,0)--(0,32);
  path d = (-40, -32)--(40, 32);
  draw (a1, blue); draw(a2, 1.9pt+blue);
  draw(b1, red); draw(b2, 1.9pt+red);
  draw(d);
  label("$a$", (40, 0), E, blue);
  label("$b$", (0, 32), N, red);
  label("$d$", (40, 32), NE);
  label("$x$", (-30,0), S, blue); label("$X$", (30,0),S, blue);
  label("$y$", (0,-19), E, red); label("$Y$", (0,19),E, red);
  \end{asy}
  \caption{A suitable choice of $X$ and $Y$}
  \label{fig:dTransverse}
  \end{figure}
  
  We are left with the case $(0,0,0)$. In this case, any curve in $F$ will
  satisfy the conclusion of the lemma, and hence all we have to do is to prove that $F$ is
  nonempty. If $X\cup Y\cup Z$ and $X\cup Y\cup z$ were both separating,
  then so would be $c=z\cup Z$, by Fact~\ref{fact:TroisArcsSep}.
  Hence, among these two curves, at least one is nonseparating, and $F$ is
  nonempty (in fact, it contains at least 4 elements).
\end{proof}

\subsection{One or two-sided curves, and extra bouquets}\label{ssec:xB}

In this last paragraph of this section, we will see how to recognize, from
the graph strucutre of $\TransFin(\Sigma)$, some additional configurations.
We insist that the work in this paragraph is useful only in the case when
$\Sigma$ is non orientable; it is needed in order to make our proof of
Theorem~\ref{thm:AutC1} work in that case (see Remark~\ref{rmk:Klein} below).

We start with a simple characterization of one-sided and two-sided curves.
\begin{itemize}
\item[$\bullet$] $\mathrm{Two}(a)$ the property that for all $b$ such that $T(a,b)$
  holds, there exists $c$ such that $T(b,c)$ and $D(a,c)$ both hold,
\item[$\bullet$] $\mathrm{One}(a)$ the negation of $\mathrm{Two}(a)$:
  there exists $b$ such that $T(a,b)$ and such that there does not
  exist $c$ satisfying $T(b,c)$ and $D(a,c)$.
\end{itemize}

\begin{observation}\label{obs:1-sided}
  Let $a$ be a vertex of $\TransFin(\Sigma)$. Then the curve $a$ is
  one-sided, if and only if $\mathrm{One}(a)$ holds.
\end{observation}
\begin{proof}
  If $a$ is two-sided, and $b$ satisfying $T(a,b)$, by pushing $a$ 
  aside we find another curve $c$ as in the definition of $\mathrm{Two}(a)$. 
  This proves the revers implication.
  
  If $a$ is one-sided, let $b$ be a curve obtained by pushing $a$ aside.
  We have $T(a,b)$, and $a$ and $b$ bound a disk. Any curve $c$
  disjoint from $a$, and with $T(b,c)$, has to enter this disk: but then,
  it has to get out, which is impossible without touching $a$ and
  without intersecting $b$ another time.
\end{proof}

Our next objective is to characterize when two one-sided curves $a$ and $b$
meet exactly once, non-transversely. We will do this in several steps.
\begin{lemma}\label{lem:acapbnonconnexe}
  Let $a$, $b$ be one-sided simple curves of $\Sigma$. Suppose the intersection
  $a\cap b$ is not connected.
  Then there exists a vertex $c$ of $\TransFin(\Sigma)$,
  distinct from $a$ and $b$,
  such that for every neighbor $d$ of both $a$ and $b$ in this graph,
  and such that $D(a,d)$ or $D(b,d)$ (or both), the vertices $c$ and $d$ are
  neighbors in this fine graph.
\end{lemma}
\begin{proof}
  Let $x$ be a subarc of $b$, whose endpoints lie in $a$, and disjoint from $a$
  otherwise. Since $a\cap b$ is disconnected, such an arc exists, and has two
  distinct end points, $p$ and $q$. Let $x'$ and $x''$ be the two subarcs of $a$
  whose ends are $p$ and $q$.
  From Fact~\ref{fact:TroisArcsSep}, we know that $x\cup x'$ or $x\cup x''$
  (or both) is a nonseparating curve; denote it by $c$. By construction,
  we have $c\neq a$ and $c\neq b$.
  
  Now let $d$ be a curve satisfying the hypothesis of the lemma. If $d$ is
  disjoint from $a$ and $b$, then it is disjoint from $c$; otherwise $d$
  intersects exactly one of $a$, $b$, far away from the other. So the intersection
  between $d$ and $c$, if any, is still transverse, and $d$ is a neighbor of
  $c$ in~$\TransFin(\Sigma)$.
\end{proof}
This contrasts with the situation we want to characterize, as we see now.
\begin{lemma}\label{lem:I(a,b)-lemme1}
  Let $a$ and $b$ be two one-sided curves, and suppose that $a\cap b$
  consists in one, inessential intersection point. Then, for every
  nonseparating curve $c\not\in\{a,b\}$, there exists $d$ such
  that $T(a,d)$ and $D(b,d)$
  hold but such that the intersection $c\cap d$ is infinite.
\end{lemma}
\begin{proof}
  We first observe that $\Sigma'=\Sigma\smallsetminus(a\cup b)$ is connected.
  This is seen by following the curves $a$ and $b$ in both directions: the
  union $a\cup b$ does not disconnect its small neighborhoods.
  Let $c$ be a curve as above. Then, we may consider a first curve $d_0$,
  obtained by pushing $a$ aside, in such a way that $d_0$ is disjoint from
  $b$ (this is possible since the intersection $a\cap b$ is inessential).
  Since $c\not\in\{a,b\}$, the curve $c$ intersects $\Sigma'$.
  Since $d_0$ meets every component of $\Sigma'$ (there is only one), we may
  deform it into a curve $d$ which intersects $c$ infinitely many times,
  exactly as in the proof of Lemma~\ref{lem:CestLeBouquet}.
\end{proof}
After these two lemmas, we have a simple sentence in terms of the graph
$\TransFin(\Sigma)$, which holds when $a\cap b$ is a single inessential
intersection point, and which guarantees that $a\cap b$ is connected.
In order to upgrade this into a characterization of the first situation,
we need to be able to exclude as well the cases when $a\cap b$ is a
non degenerate arc. These cases fall into two subcases: the intersection
arc $a\cap b$ can be essential or inessential, exactly as an intersection
point. One way to formalize this, is to say that
the intersection $a\cap b$ is essential if $a$ cuts
a regular neighborhood of $a\cap b$ into two regions both containing a
subarc of $b$, and inessential otherwise.
\begin{lemma}\label{lem:ArcEssentiel}
  Let $a$ and $b$ be one-sided curves.
  Suppose that $a\cap b$ is a non degenerate arc, and suppose this intersection
  is essential. Then there exist curves $\alpha,\alpha',\beta,\beta'$ obtained
  by pushing $a$ aside, such that $B(a,\alpha,\beta)$, $B(b,\alpha,\beta)$,
  $B(a,\alpha',\beta')$, $B(b,\alpha',\beta')$, and $N(a,\alpha,\beta')$.
\end{lemma}
\begin{proof}[Proof of Lemma~\ref{lem:ArcEssentiel}]
  The curves $\alpha$, $\beta$, $\alpha'$ and $\beta'$ may be taken in a
  neighborhood of $a\cup b$, as pictured in Figure~\ref{fig:PropI-1}.
  \begin{figure}[htb]
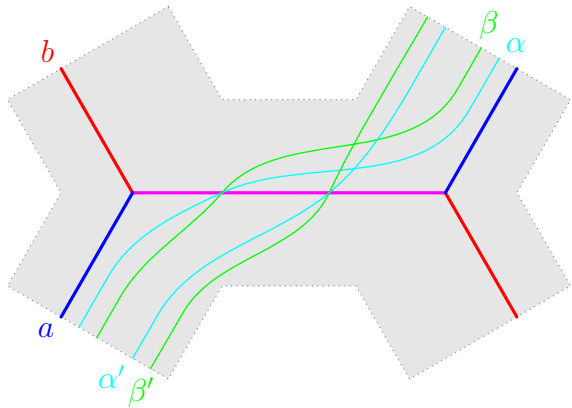

  \begin{asy}
  import geometry;
  
  point A = (0,70), B = (-60,35), C = (-60,-35), D = (0, -70),
    E = (90, -70), F = (150, -35), G = (150, 35), H = (90, 70);
  point O1 = (-40/3,0), O2 = (90+40/3,0);
  point U = (-40,0), V = (130, 0);
  path contour = A--B--(-40,0)--C--D--(20,-35)--(70,-35)--
    E--F--(130,0)--G--H--(70, 35)--(20,35)--cycle;

  path alpha = ((5*C+4*D)/9)..((5*U+4*D)/9){dir(60)}..(20,0){dir(25)}..
    ((5*V+4*H)/9){dir(60)}..((5*G+4*H)/9);
  path beta = ((4*C+5*D)/9)..((4*U+5*D)/9){dir(60)}..(20,0){dir(50)}..
    ((4*V+5*H)/9){dir(60)}..((4*G+5*H)/9);

  path alphaprim = ((2*C+7*D)/9)..((2*U+7*D)/9+10*dir(60)){dir(60)}..
    (60,0){dir(40)}..((2*V+7*H)/9){dir(60)}..((2*G+7*H)/9);
  path betaprim = ((1*C+8*D)/9)..((1*U+8*D)/9+20*dir(60)){dir(60)}..
    (60,0){dir(65)}..((1*V+8*H)/9){dir(60)}..((1*G+8*H)/9);

  fill (contour, lightgray);
  draw (contour, gray+dotted);

  draw (O1--O2, magenta+1.2pt);
  draw ((A+2*B)/3--O1, red+1.2pt); label("$b$", (A+2*B)/3, NW, red);
  draw ((2*C+D)/3--O1, blue+1.2pt); label("$a$", (2*C+D)/3, SW, blue);
  draw ((E+2*F)/3--O2, red+1.2pt); draw ((2*G+H)/3--O2, blue+1.2pt);

  draw (alpha, cyan);  label("$\alpha$", (5*G+4*H)/9, NE, cyan);
  draw(beta, green);  label("$\beta$", (4*G+5*H)/9, NNE, green);
  draw(alphaprim, cyan);  label("$\alpha'$", (2*C+7*D)/9, SW, cyan);
  draw(betaprim, green);  label("$\beta'$", (1*C+8*D)/9, SSW, green);
  \end{asy}
  \caption{The curves $\alpha,\beta,\alpha',\beta'$}
  \label{fig:PropI-1}
  \end{figure}
\end{proof}
Finally, we deal with inessential arcs.
\begin{lemma}\label{lem:IntNonEssentielle}
  Let $a$ and $b$ be one-sided curves, such that $a\cap b$ is
  an inessential arc or intersection point. Let $c$ be such that $T(a,c)$.
  Then there exists $d$ such that $B(a,c,d)$ and $D(b,d)$, if and only
  if the intersection point $a\cap c$ does not belong to~$b$.
\end{lemma}
\begin{proof}
  If $a\cap c$ belongs to $b$, then obviously, any curve $d$ satisfying
  $B(a,b,d)$ will meet $b$. Otherwise, we may push $a$ aside, in order to
  obtain a curve $d$, as in Figure~\ref{fig:I-xB}, left.
  %
\end{proof}
Putting all together, this yields the following characterization of inessential
intersection points between one-sided curves.
\begin{corollary}\label{cor:I(a,b)}
  Let $a$ and $b$ be one-sided curves. Then, $a\cap b$ consists of one,
  inessential intersection point, if and only if the following conditions are
  satisfied:
  \begin{enumerate}
  \item $a$ and $b$ are not neighbors in $\TransFin(\Sigma)$,
  \item for every $c\not\in\{a,b\}$, there exists $d$ such that $d$ is a
    neighbor of $a$ and $b$ in $\TransFin(\Sigma)$, and $D(a,d)$ or $D(b,d)$
    (or both), but $d$ is not a neighbor of $c$ in that graph,
  \item there do not exist $\alpha,\beta,\alpha',\beta'$ such that
    $B(a,\alpha,\beta)$, $B(b,\alpha,\beta)$, $B(a,\alpha',\beta')$,
    $B(b,\alpha',\beta')$ and $N(a,\alpha,\beta')$ all hold,
  \item there do not exist $c_1,c_2$ with $D(c_1,c_2)$ and with the following
    property:
    for $i=1,2$ we have: $T(a,c_i)$ and for all $d$, $B(a,c_i,d)$ and $D(b,d)$
    do not both hold.
  \end{enumerate}
\end{corollary}
This enumeration of conditions expressed only in terms of the graph structure
of $\TransFin(\Sigma)$, with the addition of the conditions $\One(a)$ and
$\One(b)$, will be also denoted by $I(a,b)$, for {\em inessential
intersection} (of one-sided curves).
\begin{proof}
  First, let us check that if $a$ and $b$ have one, inessential intersection
  point then $I(a,b)$ holds. Condition~(1) holds by definition,
  and~(2) follows from Lemma~\ref{lem:I(a,b)-lemme1}.
  The negation of condition~(3) would imply that the cardinal of $a\cap b$ is at least~2.
  Indeed, $B(a,\alpha,\beta)$ implies that $\alpha\cap\beta$ is a point lying
  in $a$. Thus, the bouquet conditions imply that both $\alpha\cap\beta$ and
  $\alpha'\cap\beta'$ lie in $a\cap b$. And the condition $N(a,\alpha,\beta')$
  then implies that $a\cap\alpha$ and $a\cap\beta'$ are disjoint, hence
  the two points $\alpha\cap\beta$ and $\alpha'\cap\beta'$ are distinct.
  Finally, condition~(4) follows from Lemma~\ref{lem:IntNonEssentielle}.
  Indeed, this lemma implies that the two curves $c_1$ and $c_2$ should
  both contain a point of $a\cap b$, hence they cannot be disjoint.
  
  Now, let $a$ and $b$ be any two nonseparating curves and suppose that $I(a,b)$.
  By conditions~(1) and~(2), the intersection $a\cap b$ is non empty, and
  connected. Along the lines of the proof of Lemma~\ref{lem:ArcEssentiel},
  we can see that condition~(3) implies that $a\neq b$, so $a\cap b$
  is an inessential intersection point, or an arc. Suppose for contradiction
  that it is a nondegenerate arc. By Lemma~\ref{lem:ArcEssentiel} and
  condition~(3), this intersection arc cannot be essential.
  Now Figure~\ref{fig:I-xB}, right,
  shows the desired contradiction with condition~(4).
  \begin{figure}[htb]
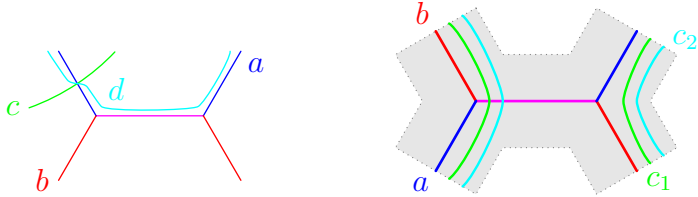

  \hfill
  \begin{asy}
  import geometry;

  point A = 28*dir(120), B = 14*dir(120), C = (0,0), D = (40,0),
    E = D+28*dir(60);
  path d = A+(-4,0)..tension 3 ..(B-4*dir(-20))..B..
    (B+4*dir(-20))..tension 3 ..(C+4*dir(60))..tension 4 ..
    (D+4*dir(120))..tension 3 ..(E+(-4,0));
  draw (C--D, magenta);
  draw (A--C, blue);  draw(D--E, blue);
  draw (C--(28*dir(-120)), red); draw(D--((40,0)+28*dir(-60)), red);
  draw ((-25,3)..B..(7,24), green);

  draw (d, cyan);

  label ("$b$", -28*dir(60), W, red); label("$a$", E, SE, blue);
  label("$c$", (-25,3), W, green); label("$d$", 4*dir(60), NE, cyan);
  \end{asy}
  \hfill
  \begin{asy}
  import geometry;
  
  point A = (0,70)/2, B = (-60,35)/2, C = (-60,-35)/2, D = (0, -70)/2,
    E = (90, -70)/2, F = (150, -35)/2, G = (150, 35)/2, H = (90, 70)/2;
  point O1 = (0,0)/2, O2 = (90,0)/2;
  point U = (-40,0)/2, V = (130, 0)/2;
  path contour = A--B--U--C--D--((20,-35)/2)--((70,-35)/2)--
    E--F--V--G--H--((70, 35)/2)--((20,35)/2)--cycle;
  
  fill (contour, lightgray);
  draw (contour, gray+dotted);

  draw (O1--O2, magenta+1pt);
  draw ((A+B)/2--O1, red+1pt); label("$b$", (A+B)/2, NW, red);
  draw ((C+D)/2--O1, blue+1pt); label("$a$", (C+D)/2, SW, blue);
  draw ((E+F)/2--O2, red+1pt); draw ((G+H)/2--O2, blue+1pt);
  
  draw (((4*A+2*B)/6)..tension 2 ..(5,0)..tension 2 ..
    ((4*D+2*C)/6), green+0.9pt);
  draw (((5*A+B)/6)..tension 2 ..(10,0)..tension 2 ..
    ((5*D+C)/6), cyan+0.9pt);
  
  draw (((4*F+2*E)/6)..tension 2 ..(55,0)..tension 2 ..
    ((4*G+2*H)/6), green+0.9pt);
  draw (((5*F+E)/6)..tension 2 ..(60,0)..tension 2 ..
    ((5*G+H)/6), cyan+0.9pt);
  
  label("$c_1$", ((4*F+2*E)/6), SSE, green);
  label("$c_2$", ((5*G+H)/6), ENE, cyan);
  \end{asy}
  \hspace{3cm}
  \caption{Some configurations of curves for properties $I$ and $xB$}
  \label{fig:I-xB}
  \end{figure}
\end{proof}

Finally, we deal with extra bouquets.
We denote by $xB(a,b,c)$ the property that
$T(a,b)$, $T(a,c)$, $I(b,c)$ all hold and moreover:
for all $d$ such that $B(a,b,d)$ holds, $D(c,d)$ does not.
\begin{lemma}
  Let $a$, $b$ and $c$ be such
  that $T(a,b)$, $T(a,c)$ and $I(b,c)$, with $b$ and $c$ one-sided.
  Then $xB(a,b,c)$ holds if and only if the intersection points
  $a\cap b$, $a\cap c$ and $b\cap c$ coincide.
\end{lemma}
\begin{proof}
  Of course if these points coincide, then property $xB(a,b,c)$
  holds: every curve $d$ such that $B(a,b,d)$ holds, must contain
  this point and hence cannot be disjoint from $c$.
  
  Now suppose that these points do not coincide, hence, are three
  pairwise distinct points. Then, we may push $b$ aside, in order
  to find a curve $d$ which does not intersect $c$ any more, as
  the intersection $b\cap c$ is not essential. This curve $d$,
  obtained by pushing $b$, can be made to satisfy $T(b,d)$,
  while crossing $b$ precisely at the point $a\cap b$, and
  this intersection can be made transverse;
  the illustration
  of this situation is similar to Figure~\ref{fig:I-xB}, left,
  and this time we leave it to the reader.
  This yields a curve $d$
  such that $B(a,b,d)$ holds and $d$ disjoint from $c$.
\end{proof}

\section{Proof of Theorem~\ref{thm:AutC1}}\label{sec:AutHomeo}

Here as above, $\Sigma$ is a connected surface admitting a nonseparating
closed curve.

\subsection{From bijections to homeomorphisms}
In order to prove Theorem~\ref{thm:AutC1}, it suffices to prove that every
automorphism of $\TransFin(\Sigma)$ is supported by a bijection of the
surface, in virtue of the following observation.
\begin{proposition}\label{prop:BijHomeo}
  Let $f\colon\Sigma\to\Sigma$ be a bijection. We suppose that for every
  nonseparating simple closed curve $\alpha\subset\Sigma$, the sets $f(\alpha)$
  and $f^{-1}(\alpha)$ are also nonseparating simple closed curves in $\Sigma$.
  Then $f$ is a homeomorphism.
\end{proposition}
\begin{proof}
  If our hypothesis was that $f$ and $f^{-1}$ are closed (\textsl{i.e.}, send
  closed sets to closed sets), then $f$ would be a homeomorphism. So our
  strategy is to use our hypothesis here in a similar fashion. We need only
  prove that $f$ is continuous, the argument for $f^{-1}$ is symmetric.
  
  Let $x\in\Sigma$ and suppose that $f$ is not continuous at $x$. Then there
  exists a sequence $(x_n)_{n\geqslant 0}$ of distinct points converging
  to $x$, and a neighborhood $V$ of $f(x)$, such that for all $n$, we have
  $f(x_n)\not\in V$.
  Notice that in the open unit disk of the plane, up to homeomorphism, there
  is only one sequence of distinct points converging to the origin.
  With this in head, we may construct an embedded arc, in $\Sigma$, with one
  end at $x$, and which contains all the points $x_n$. Then we may construct a
  nonseparating simple closed curve $\alpha$ containing this arc.
  By hypothesis, $f(\alpha)$ is a nonseparating closed curve in $\Sigma$, which
  contains $f(x)$. We may perform a surgery of $f(\alpha)$ inside $V$, to
  obtain a nonseparating simple closed curve $\beta$, which coincides with
  $f(\alpha)$ outside $V$ but which does not contain $f(x)$.
  Now $f^{-1}(\beta)$ is, by hypothesis, a closed subset of $\Sigma$, which
  contains all the points $x_n$ but not $x$. This is a contradiction.
\end{proof}

\subsection{The adjacency relation $\adjacent$}\label{sec:Adjacency}

Let $E_T(\TransFin(\Sigma))$ denote the set of edges $\{a,b\}$ of
$\TransFin(\Sigma)$ satisfying $T(a,b)$. Then we have a map
\[ \Point\colon \ E_T(\TransFin(\Sigma))\to\Sigma, \]
which to each edge $\{a,b\}$, associates the intersection point $a\cap b$.
The main part of proof of Theorem~\ref{thm:AutC1}
consists in showing that we can express the
equality
\[ \Point(a,b) = \Point(\alpha, \beta) \]
in terms of the graph.
For this we introduce the equivalence relation $\adjacent$ on
$E_T(\TransFin(\Sigma))$ as follows.
Let $\{a,b\}$ in $E_T(\TransFin(\Sigma))$. If $\{a,b,c\}$
is a $3$-clique of $\TransFin(\Sigma)$ of type bouquet we
set $\{a,b\}\adjacent \{a,c\}$.
We also set
$\{a,b\}\adjacent \{a,c\}$
if $a$, $b$ and $c$ satisfy the ``extra bouquet'' condition, denoted
above by $xB(a,b,c)$, see paragraph~\ref{ssec:xB} (this is void
when $\Sigma$ is orientable).
Then $\adjacent$ is defined as the equivalence relation generated by these
relations.
When $\Sigma$ is orientable, the relation $\adjacent$ corresponds to the
equivalence relation on triangles, generated by adjacency, in the subgraphs
of $\TransFin(\Sigma)$ induced by curves passing through a common point.
This is what motivates our notation.

The relation $\{a,b\}\adjacent \{a', b'\}$ obviously implies
$\Point(a,b) = \Point(a',b')$. We will see that the converse is not true,
and describe geometrically the equivalence classes in
Section~\ref{sec:LocalSubgraphs}, but for now we will only need the
following partial statement.
\begin{proposition}\label{prop:GermeAdjacent}
  Let $a$, $b$, $a'$, $b'$ be such that $T(a,b)$ and $T(a',b')$.
  Suppose that they have the same intersection point,
  $x=\Point(a,b)=\Point(a',b')$, and suppose that the germs of $a$ and $a'$
  coincide, \textsl{i.e.}, there exists a neighborhood $V$ of $x$ such
  that $a\cap V=a'\cap V$. 
  Then $\{a,b\}\adjacent\{a',b'\}$.
\end{proposition}
\begin{remark}\label{rmk:Klein}
  If $\Sigma$ is a Klein bottle, there are no couples $\{a,b\}$ of two-sided
  curves such that $T(a,b)$, and for every one-sided curve $b$, the curves
  $c$ such that $T(b,c)$ holds fall into only two isotopy classes: that of $b$
  and that of a two-sided curve, prescribed by $b$. It follows that, without
  the extra bouquets in the definition of $\adjacent$, there would
  have been too many classes of $\adjacent$, as such a class would remember
  the isotopy class of a one-sided curve, and Proposition~\ref{prop:GermeAdjacent}
  would not be true in this special case. These extra bouquets will be used
  in the proof of Lemma~\ref{lem:SpheresConnexes} below.
\end{remark}
We postpone the proof of the proposition to the end of this section;
for now we will explain how it implies Theorem~\ref{thm:AutC1}.

\subsection{Proof of Theorem~\ref{thm:AutC1}}
If $a,b,c$ are vertices of $\TransFin(\Sigma)$, we denote by $F(a,b,c)$ the
property that $T(a,b)$ holds, and there exists an edge $\{a',b'\}$ with
$\{a,b\}\adjacent\{a',b'\}$ such that $\{a',b',c\}$ is a $3$-clique which is
not of type bouquet.
Note that this property $F(a,b,c)$ implies that $c$ does not contain the
point $\Point(a',b') = \Point(a,b)$. The next lemma asserts that $F(a,b,c)$
actually characterises this geometric property, and the letter $F$ stands for:
``$a\cap b$ is far from $c$''.
\begin{lemma}\label{lem:AppartenanceGraphe}
  Let $a,b,c$ be vertices of $\TransFin(\Sigma)$, and suppose $T(a,b)$ holds.
  Then
  \[ F(a, b, c) \Leftrightarrow \Point(a,b)\not\in c. \]
\end{lemma}

\begin{proof}
  The direct implication follows directly from the
  definitions; we have to prove
  the converse implication. Suppose $\Point(a,b)\not\in c$.
  Since $\Sigma$ is connected, there exists a regular neighborhood of $c$
  containing the point $\Point(a,b)$. Depending on whether $c$ is one-sided or
  two-sided, up to homeomorphism, this leads to only two distinct situations.
  In Figure~\ref{fig:OnComplete2germes}, we represent in bold the germs of
  the curves $a$ and $b$ near the point $a\cap b$, and show how to complete
  these germs to new curves $a'$, $b'$ such that $\{a',b',c\}$ is a $3$-clique
  not of type bouquet.
  When $c$ is one-sided (see Figure~\ref{fig:OnComplete2germes}, left),
  we may use two curves $a'$, $b'$ obtained by pushing $c$, while when $c$
  is two-sided (see Figure~\ref{fig:OnComplete2germes}, right), we have to
  use a curve $d$ which meets $c$ once transversely.
  
  Now, by Proposition~\ref{prop:GermeAdjacent}, we have
  $\{a,b\}\adjacent\{a',b'\}$, and $\{a',b',c\}$ is a non-bouquet $3$-clique,
  so we have $F(a,b,c)$ by definition.
  \begin{figure}[htb]
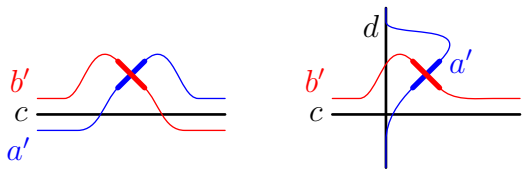

  \begin{asy}
    import geometry;
  
    picture commun, gauche, droite;
    path agros = (0, 10)--(10, 20), bgros = (0, 20)--(10, 10), c = (-30, 0)--(40,0);
    
    draw (commun, c, 1pt+black); draw(commun, agros, 2pt+blue);
    draw (commun, bgros, 2pt+red);
    add(gauche, commun); add(droite, commun);
    
    path b1fingauche = (-30, 6){dir(0)}..(-20, 6){dir(0)}..(0,20){dir(-45)};
    path b2fingauche = (10, 10){dir(-45)}..{dir(0)}(25, -6)--(40, -6);
    path a1fingauche = (-30, -6)--(-15,-6){dir(0)}..(0,10){dir(45)};
    path a2fingauche = (10, 20){dir(45)}..{right}(30, 6)--(40,6);
    draw (gauche, b1fingauche, red); draw(gauche, b2fingauche, red);
    draw (gauche, a1fingauche, blue); draw(gauche, a2fingauche, blue);
    label(gauche, "$a'$", (-30,-6), SW, blue);
    label(gauche, "$c$", (-30,0), W);
    label(gauche, "$b'$", (-30,6), NW, red);
    
    path d = (-10, -20)--(-10, 40);
    draw (droite, d, 1pt+black);
    draw (droite, b1fingauche, red);
    path b2findroite = (10,10){dir(-45)}..{right}(30, 6)--(40,6);
    draw (droite, b2findroite, red);
    path a1findroite = (-10,-20)--(-10,-10){up}..(0,10){dir(45)};
    path a2findroite = (10,20){dir(45)}..{up}(-10,35)--(-10,40);
    draw (droite, a1findroite, blue); draw(droite, a2findroite, blue);
    label(droite, "$c$", (-30, 0), W);
    label(droite, "$b'$", (-30,6), NW, red);
    label(droite, "$a'$", (10,20), E, blue);
    label(droite, "$d$", (-10,40), SW);
    
    add(gauche, (-50, 0)); add(droite, (60,0));
  \end{asy}
  \caption{Completing the germs of $a$ and $b$ to form a non-bouquet $3$-clique $(a',b',c)$}
  \label{fig:OnComplete2germes}
  \end{figure}
\end{proof}
\begin{corollary}\label{cor:AppartenanceGraphe}
  Suppose $T(a,b)$ and $T(\alpha,\beta)$ hold.
  Then $\Point(a,b)\neq\Point(\alpha,\beta)$ if and only if there
  exists a nonseparating closed curve $c$ such that 
  $F(\alpha, \beta, c)$ holds but not $F(a,b,c)$. 
\end{corollary}
The corollary is a direct consequence of Lemma~\ref{lem:AppartenanceGraphe}.
It follows that the equality $\Point(a,b)=\Point(\alpha,\beta)$ can be
expressed in terms of the graph structure of $\TransFin(\Sigma)$, because,
as a consequence of Corollary~\ref{coro:TypesPrecis}, being a $3$-clique
not of type bouquet is also characterized in terms of this graph structure.
Now we can conclude the proof of Theorem~\ref{thm:AutC1}, provided
Proposition~\ref{prop:GermeAdjacent} holds.
\begin{proof}[Proof of Theorem~\ref{thm:AutC1}]
  Let $\varphi$ be an automorphism of $\TransFin(\Sigma)$.
  Given a point $x$ in $\Sigma$, we choose two nonseparating simple closed
  curves $a$, $b$ intersecting exactly once, transversely, at $x$,
  and set $\varphi_\Sigma(x)=\Point(\varphi(a),\varphi(b))$. This
  formula is valid because, by Proposition~\ref{prop:Types},
  $\varphi(a)$ and $\varphi(b)$ are still nonseparating simple closed
  curves intersecting exactly once. The point $\varphi_\Sigma(x)$
  does not depend on the choice of $(a,b)$, because if $(\alpha,\beta)$
  is another choice, the equalities $\Point(a,b)=\Point(\alpha,\beta)$
  and $\Point(\varphi(a),\varphi(b))=\Point(\varphi(\alpha),\varphi(\beta))$
  can be all expressed in terms of the graph structure of
  $\TransFin(\Sigma)$, thanks to Corollary~\ref{cor:AppartenanceGraphe}.
  Thus, the map $\varphi_\Sigma$ is well-defined, and by following the
  definitions we observe that the map
  $(\varphi^{-1})_\Sigma$ is its inverse: hence $\varphi_\Sigma$
  is a bijection of $\Sigma$. Finally, it follows from Lemma~\ref{lem:AppartenanceGraphe}
  that for any nonseparating simple closed curve $\alpha$, the curve
  $\varphi(\alpha)$ coincides with the set of points $\varphi_\Sigma(x)$
  as $x$ describes $\alpha$. In other words, the automorphism $\varphi$
  is realized by the bijection $\varphi_\Sigma$. Proposition~\ref{prop:BijHomeo}
  concludes.
\end{proof}

\subsection{Connectedness of some arc graphs}
In order to finally prove
Proposition~\ref{prop:GermeAdjacent},
we will first need a couple of elementary results on
fine arc graphs.

\begin{lemma}\label{lem:ArcConnexe}
  Let $S$ be a connected topological surface, with boundary, and let $x,y$ be
  two distinct points of $\partial S$. Let $\mathcal{E}\Afin{}(S, x, y)$ be the graph whose
  vertices are the simple arcs joining $x$ and $y$ and which meet $\partial S$ only at
  their ends, with an edge between two such arcs if and only if they are
  disjoint except at $x$ and $y$. Then the graph $\mathcal{E}\Afin{}(S, x, y)$ is connected.
\end{lemma}
Note that, when $x$ and $y$ are taken in the same connected component, we are
not requiring that the arcs be nonseparating; this is the reason why we use the
letter $\mathcal{E}$, for {\em extended}, in the same fashion as in~\cite{LMPVY}.
\begin{proof}
  Let $a$, $b$ be two vertices of this graph. As a first case we suppose that
  $a\cap b$ is made of a finite number of transverse intersection points:
  we will prove by induction on the
  cardinal of $a\cap b$ that, in this case, $a$ and $b$ are connected in
  $\mathcal{E}\Afin{}(S,x,y)$. If $a\cap b$ is as small as possible, \textsl{i.e.}, is
  equal to $\lbrace x,y\rbrace$, then $a$ and $b$ are neighbors in this graph.
  Otherwise, if $a$ and $b$ intersect at other points than $x$ and $y$, we
  may, in the spirit of~\cite{HPW},
  pick a {\em unicorn path} $c$, made of one subarc of $a$ beginning at $x$,
  and one subarc of $b$ ending at $y$. (For example, we may follow $a$
  until it first meets $b$ after $x$, and then
  continue along $b$).
  Now we may push $c$ aside while fixing its ends,
  at the appropriate side of $c$, to
  obtain a new arc $c'$
  with both $c'\cap a$ and $c'\cap b$ of cardinal strictly lower than that of
  $a\cap b$.

  This proves that arcs intersecting at finitely many points are
  connected in $\mathcal{E}\Afin{}(S,x,y)$.
  
  Now if the intersection $a\cap b$ is infinite, fix a differentiable structure
  on the surface $S$. We may consider a smooth curve $a'$, neighbor of $a$, by
  pushing $a$ aside while fixing its ends, and similarly, a smooth neighbor
  $b'$ of $b$ similarly. Up to perturbing $b'$, we may suppose that $b'$ is
  transverse to $a'$. 
  By the step above, $a'$ and $b'$ are
  connected in $\mathcal{E}\Afin{}(S,x,y)$ and the Lemma is proved.
\end{proof}
We will also need a version for nonseparating arcs.
\begin{lemma}\label{lem:ArcConnexeBis}
  Let $S$ be a connected surface containing a nonseparating curve. Suppose
  $S$ has boundary, and let $x$ and $y$ be two distinct points in a same boundary
  component of $S$. Let $\NAfin{}(S,x,y)$
  be the set of nonseparating simple arcs
  connecting $x$ to $y$, with an edge when they are disjoint away from $x$
  and $y$. Then $\NAfin{}(S,x,y)$ is connected.
\end{lemma}
The points $x$ and $y$ add some technicality; let us first prove the following
simpler statement.
\begin{lemma}\label{lem:NonSepFinRelUnBordConnexe}
  Let $S$ be a connected surface containing a nonseparating curve, and with
  at least one boundary component, denoted $C$. Let $\NAfin{}(S,C)$ be the
  graph whose vertices are the nonseparating arcs joining two distinct points
  of $C$, and with an edge between two such vertices whenever they are disjoint.
  Then this graph is connected.
\end{lemma}
This lemma is a variation on~\cite[Corollary~3.2]{LMPVY}; here we additionnaly
require that the arcs end at $C$.
In fact, in~\cite{LMPVY}, Corollary~3.3 is stated for surfaces with
$b>0$ boundary components, but proved only in the case $b=1$, which is the case needed in the proof of their main theorem.
Lemma~\ref{lem:NonSepFinRelUnBordConnexe} may be used to extend this corollary
to any $b>0$.

\begin{proof}
  We begin with the observation that the graph $\NAfin{}(S,C)$ has no
  isolated point. Indeed, if $\gamma$ is a vertex of $\NAfin{}(S,C)$,
  by definition it is nonseparating. So we may consider a simple closed curve
  $u$ with one, transverse intersection point with~$\gamma$. Obviously,
  this curve $u$ is nonseparating; this follows from Fact~\ref{fact:NonSep},
  applied to $u$, and a curve $v$ obtained by concatenation of $\gamma$
  with some arc of $C$. Now we can perform a surgery on $u$,
  and push its intersection point towards one end of $\gamma$ until we hit
  $C$. This constructs an arc
  $\alpha$,
  which is now disjoint from $\gamma$, and which is also nonseparating.

  Next, we claim that we can suppose, without loss of generality, that
  the surface $S$ is compact. Indeed, if $\gamma_1,\gamma_2$ are vertices
  of $\NAfin{}(S,C)$, and if $u$ is a nonseparating curve intersecting $\gamma_1$
  as above, consider the set $K=C\cup\gamma_1\cup\gamma_2\cup u$. This set is
  compact, hence there exists a compact topological subsurface $S'$ of $S$
  containing~$K$.
  This surface $S'$ contains nonseparating curves,
  as it contains $u$ and $\gamma_1\cup C$, which may be used as above to
  find two simple closed curves $u$, $v$ with one, essential intersection.
  Now a path joining $\gamma_1$ to $\gamma_2$ in $S'$ is also a path joining
  $\gamma_1$ to $\gamma_2$ in $S$.
  So, until the end of the proof, $S$ is now supposed to be compact.
  
  Next, observe that if two vertices $\gamma_1,\gamma_2$ of $\NAfin{}(S,C)$ are
  isotopic (\textsl{i.e.}, there exists a continuous map $H\colon[0,1]^2\to S$
  such that $\gamma_1(t)=H(0,t)$ and $\gamma_2(t)=H(1,t)$ for all $t$,
  $H(s,0),H(s,1)\in C$ for all $s$ and the curve $H_s\colon t\to H(s,t)$ is
  injective for all $s$),
  then $\gamma_1$ and $\gamma_2$ are in the same component of $\NCfin{}(S,C)$.
  This argument is borrowed from~\cite{BHW}: for all $s$, the arc
  $H_s$ has at least a neighbor $\alpha_s$ (by the first observation above),
  and the set of $s'$ such that $H_{s'}$ is still a neighbor of $\alpha_s$ is
  open in $[0,1]$. By compactness of $[0,1]$, there exist a finite number of
  arcs $\alpha_1,\ldots,\alpha_n$, and a subdivision $0=t_0<t_1<\cdots<t_n=1$
  such that $\alpha_j$ is disjoint from $H_t$ for all $t\in[t_{j-1},t_j]$ for
  all $j$, and now $(\gamma_1,\alpha_1,\ldots,\alpha_n,\gamma_2)$ is a
  path of $\NCfin{}(S,C)$ joining $\gamma_1$ to $\gamma_2$.
  As a result of this observation, we need only prove the connectedness of
  the graph $\mathrm{NA}(S,C)$, whose vertices are the isotopy classes of
  arcs between two distinct points of $C$, and with an edge between two
  vertices whenever the corresponding classes admit disjoint representatives.
  
  We proceed with the observation that the graph $\mathrm{A}(S,C)$, defined
  exactly as $\mathrm{NA}(S,C)$ excpept we consider essential arcs, which may
  be separating, is connected.
  A simple way to do this is by using the idea of unicorn arcs exactly as
  in the proof of the preceding lemma: if two arcs $a$ and $b$ are in
  minimal position then their unicorn arcs are essential, and have fewer
  intersections with both $a$ and $b$ than the number of points of $a\cap b$.
  
  We will promote the connectedness of $\mathrm{A}(S,C)$ to that of
  $\mathrm{NA}(S,C)$, by induction on the number of boundary components of $S$.
  
  First, suppose that $S$ has only one boundary component,~$C$.
  Let $\gamma_1,\gamma_2$ be two vertices of $\mathrm{NA}(S,C)$.
  We may connect them by a path
  $(\gamma_1,\alpha_1,\alpha_2,\ldots,\alpha_n,\gamma_2)$
  in $\mathrm{A}(S,C)$, where each of the $\alpha_j$ may be separating;
  consider such a path with minimal number of separating arcs.
  For contradiction, and up to some relabeling, suppose $\alpha_1$ is
  separating. Then it cuts $S$ in two components; denote by $S_1$ the
  one containing $\gamma_1$
  and $S_2$ the other. Then $\alpha_2$ is also contained in $S_1$,
  otherwise we may delete $\alpha_1$ from our path. Since $S$ has
  no boundary component other than $C$ and since the curve $\alpha_1$
  is essential, the surface $S_2$ contains
  a nonseparating arc, $\alpha_1'$. This arc
  may be used instead of $\alpha_1$ in our initial path
  from $\gamma_1$ to $\gamma_2$, contradicting the minimality of the
  number of separating arcs. This proves that
  $\mathrm{NA}(S,C)$ is connected if $S$ has no other boundary
  component.
  
  Now, we suppose, for inductive hypothesis, that
  $\mathrm{NA}(S',C')$
  is connected for every surface $S'$ with less boundary components
  than $S$. Let $\gamma_1,\gamma_2$ be two vertices of $\mathrm{NA}(S,C)$.
  As before, consider a path
  $(\gamma_1,\alpha_1,\alpha_2,\ldots,\alpha_n,\gamma_2)$ in
  $\mathrm{A}(S,C)$ between them, with minimal number of separating arcs.
  For contradiction, and up to some relabeling, suppose $\alpha_1$ is
  separating: it cuts $S$ into two subsurfaces, let $S_1$ be the one
  containing $\gamma_1$, and, by hypothesis, must also contain $\alpha_2$,
  and let $S_2$ be the other. If $S_2$ contains nonseparating arcs, we
  conclude as before. If not, then $S_2$ contains some of the boundary
  components of $S$, hence the surface with boundary $S'=S_1\cup\alpha_1$
  has strictly less boundary components than $S$. One is $C'$, composed
  by an arc of $C$ and the arc $\alpha_1$, and there may be others.
  
  If $\alpha_2$ is nonseparating, then, by the induction hypothesis,
  there is a path
  $(\gamma_1,\beta_1,\ldots,\beta_k,\alpha_2)$ of $\mathrm{NA}(S',C')$
  connected them. The arcs $\beta_1, \ldots, \beta_k$ may have end points
  in $\alpha_1$, but we may perform a surgery in order to push all these
  points to $C$, and obtain arcs $\beta_1',\ldots,\beta_k'$ which are
  also vertices of $\mathrm{NA}(S,C)$, and we are done in this case.
  
  Finally, if $\alpha_2$ is a separating arc (of $S$, or of $S_1$,
  equivalently), then we may find an arc $\alpha_2'$ of $S_1$ which
  is nonseparating and disjoint from $\alpha_2$. By following the
  last case above, there exists a path
  $(\gamma_1,\beta_1,\ldots,\beta_k,\alpha_2')$ in $\mathrm{NA}(S,C)$,
  hence the path
  $(\gamma_1,\beta_1,\ldots,\beta_k,\alpha_2',\alpha_2,\ldots,\alpha_n,\gamma_2)$
  of $\mathrm{A}(S,C)$ has one less separating arc than the initial path.
  This contradiction ends the proof.
\end{proof}
\begin{proof}[Proof of Lemma~\ref{lem:ArcConnexeBis}]
  Let $\gamma_1,\gamma_2$ be two vertices of $\NAfin{}(S,x,y)$.
  First, we may construct a neighbor $\gamma_2'$ in $\NAfin{}(S,x,y)$
  of $\gamma_2$, which, in a neighborhood of $x$ (resp. $y$),
  touches $\gamma_1$ only at $x$ (resp. $y$).
  
  Indeed, there is a neighborhood $U_x$ of $x$ homeomorphic to the
  closed half unit disk
  \[ \{z, |z|\leqslant 1 \text{ and }\mathrm{Im}(z)\geqslant 0\}, \]
  where the middle ray ($\mathrm{Re}(z)=0$) corresponds to the points
  of $\gamma_1$. On either side of this ray, we may find an arc disjoint
  from $\gamma_1$ and $\gamma_2$ except at $0$, arbitrarily close to
  the boundary ($\mathrm{Im}(z)=0$), and joining $0$ to the unit circle,
  and then this small arc may be continued to construct a curve $\gamma_2'$
  which consists of pushing $\gamma_2$ aside.
  
  So we may suppose that $\gamma_1$ and $\gamma_2$, close to $x$ and
  $y$, intersect only at these points, and we may now find neighborhoods
  $U_x$ and $U_y$ as above, such that their intersections with
  $\gamma_1$ and $\gamma_2$ are along rays in this disk, in distinct
  directions around $0$. Let $S'$ be the surface obtained by removing
  the interiors of $U_x$ and $U_y$ from $S$. Then the path given by
  applying Lemma~\ref{lem:NonSepFinRelUnBordConnexe} to $S$, yields
  a path from $\gamma_1$ to $\gamma_2$ in $\NAfin{}(S,x,y)$, just by
  adding some rays in $U_x$ and $U_y$ to the corresponding arcs.
\end{proof}

\subsection{Proof of Proposition~\ref{prop:GermeAdjacent}}

Let us go back to the proof of
Proposition~\ref{prop:GermeAdjacent}.
For the remaining of the section we fix a point $x\in\Sigma$.
Let $X$ denote the set of
nonseparating simple closed curves passing through~$x$.
\begin{lemma}\label{lem:SpheresConnexes}
  Let $a,b,c\in X$.
  Suppose that $T(a,b)$ and $T(a,c)$ hold.
  
  Then $(a,b)\adjacent(a,c)$.
\end{lemma}
\begin{proof}
  Let $S$ be the surface obtained by cutting $\Sigma$ along $a$: it is
  the surface with boundary obtained by gluing back two copies of the
  curve $a$ to $\Sigma\smallsetminus a$. The point $x$ of $\Sigma$
  yields two points, $p$ and $q$, of $\partial S$, and the curves
  $b$ and $c$ define two arcs of $S$ joining $p$ and $q$.
  By Lemma~\ref{lem:ArcConnexe}, there exists a finite sequence
  $\gamma_0=b$, \ldots, $\gamma_n=c$, of arcs of $S$ joining $p$ and $q$,
  with $\gamma_i$ and $\gamma_{i+1}$ disjoint except at $p$ and $q$.
  For each $i$, the arc $\gamma_i$ defines a closed curve in $\Sigma$,
  which has precisely one, transverse intersection with $a$; we will
  still denote it by $\gamma_i$, abusively.
  
  For every $i$, if $T(\gamma_i,\gamma_{i+1})$ holds, then we have
  $(a,\gamma_i)\adjacent(a,\gamma_{i+1})$, by definition.
  If $T(\gamma_i,\gamma_{i+1})$ does not hold, then either $\gamma_i$
  or $\gamma_{i+1}$ are both one-sided, or one of them is two-sided.
  In the first case, the condition $xB(a,\gamma_i,\gamma_{i+1})$ holds,
  by definition, and hence
  $(a,\gamma_i)\adjacent(a,\gamma_{i+1})$.
  In the second, up to reversing the notation suppose $\gamma_i$ is
  two-sided. Figure~\ref{fig:OnCompleteBouquets} shows how to insert
  a curve $\delta$ such that
  $B(a,\delta,\gamma_i)$ and $B(a,\delta,\gamma_{i+1})$ both hold,
  and hence we still have $(a,\gamma_i)\adjacent(a,\gamma_{i+1})$ in
  this case.
  
  \begin{figure}[htb]
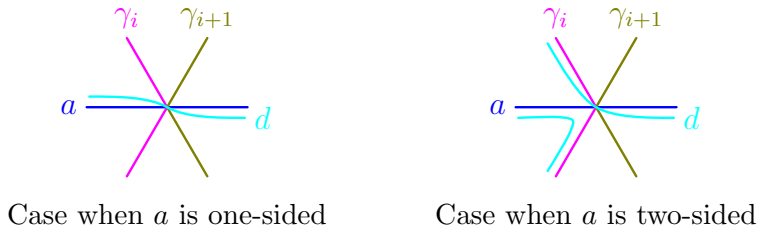

  \begin{asy}
  import geometry;
  
  picture commun, gauche, droite;
  
  path a = (-30,0)--(30,0), gamma1 = 30*dir(120)--(0,0)--30*dir(-120);
  path gamma2 = rotate(180)*gamma1;
  
  path dgauche = (-29, 4){right}::(0,0){dir(-30)}::(29, -4){right};
  path ddroit1 = (30*dir(127)){dir(-60)}::(0,0){dir(-30)}::(29,-4){right};
  path ddroit2 = (-29, -4){right}..tension 3 ..(10*dir(210)){dir(-60)}..
    tension 3 ..(30*dir(-127)){dir(-120)};
  
  draw (commun, a, blue+0.9pt);  label(commun, "$a$", (-30,0), W, blue);
  draw(commun, gamma1, magenta+0.9pt);
  label(commun, "$\gamma_i$", 30*dir(120), N, magenta);
  draw(commun, gamma2, olive+0.9pt);
  label(commun, "$\gamma_{i+1}$", 30*dir(60), N, olive);
  
  draw(gauche, dgauche, cyan+0.9pt);
  draw(droite, ddroit1, cyan+0.9pt);
  draw(droite, ddroit2, cyan+0.9pt);
  label(commun, "$d$", (29, -4), E, cyan);
  
  label(gauche, "{\small Case when $a$ is one-sided}", (0,-40));
  label(droite, "{\small Case when $a$ is two-sided}", (0,-40));
  
  add(commun, (-80,0)); add(commun, (80,0));
  add(gauche, (-80,0)); add(droite, (80,0));
  \end{asy}
  \caption{Connecting the curves by common adjacency}
  \label{fig:OnCompleteBouquets}
  \end{figure}
  
  By transitivity, we deduce that $(a,b)\adjacent(a,c)$.
\end{proof}
The last ingredient for the proof of Proposition~\ref{prop:GermeAdjacent}
is the following observation.
\begin{observation}\label{obs:CCarc}
  Let $a$, $a'$ be two nonseparating simple closed curves in $\Sigma$
  such that $a\cap a'$ is an arc. Then, both sides of this arc lie
  in the same connected component of $\Sigma\smallsetminus(a\cup a')$.
\end{observation}
\begin{proof}
  \textsl{A priori}, the complement of $\Sigma\smallsetminus(a\cup a')$
  may have up to four connected components, as suggested in 
  Figure~\ref{fig:CCarc}.
  \begin{figure}[htb]
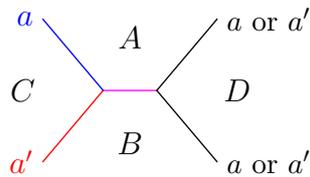

  \begin{asy}
  import geometry;
  
  real R = 35; point Dec = (20,0);
  draw ((0,0)--Dec, magenta);
  draw (shift(Dec)*(R*dir(50)--(0,0)));
  label ("{\small $a$ or $a'$}", Dec+R*dir(50), E);
  draw ((0,0)--(-R*dir(50)), red); label ("$a'$", -R*dir(50), W, red);
  draw (shift(Dec)*(R*dir(-50)--(0,0)));
  label ("{\small $a$ or $a'$}", Dec+R*dir(-50), E);
  draw ((0,0)--(-R*dir(-50)), blue); label ("$a$", -R*dir(-50), W, blue);
  label ("$A$", Dec/2+(0, 20));  label("$B$", Dec/2+(0, -20));
  label ("$C$", (-30, 0));  label ("$D$", Dec+(30, 0));
  \end{asy}
  \caption{The arc $a\cap a'$ cannot disconnect}
  \label{fig:CCarc}
  \end{figure}
  Suppose first that the intersection $a\cap a'$ is essential.
  If $a$ (resp. $a'$) is one-sided, by following the curve $a$ (resp. $a'$)
  we see that $A=B$. If both $a$ and $a'$ are two-sided, by following $a$ we
  see that $A=D$ and $C=B$, while by following $a'$ we get $A=C$ and $B=D$,
  so~$A=B$.
  
  Now, suppose the intersection arc $a\cap a'$ is inessential.
  By following $a$, we see that $C=D$, regardless of $a$ being one or
  two-sided. Thus, if $A\neq B$, then one of $A$
  or $B$, say $A$, is not connected from any of $B, C, D$. But this implies
  that $a$ is separating, a contradiction.
  \end{proof}
  
We are now in a position to prove~Proposition~\ref{prop:GermeAdjacent},
but instead we will
prove the following stronger statement, which will be more
convenient later in this article.
\begin{proposition}\label{prop:SemiGermeAdjacent}
  Let $a$, $b$, $a'$, $b'$ be such that $T(a,b)$ and $T(a',b')$,
  with intersection point
  $x=\Point(a,b)=\Point(a',b')$, and suppose that $a$ and $a'$ locally ``half 
  coincide'' near $x$, \textsl{i.e.}, $a \cap a'$ contains a non degenerate arc 
  with endpoint $x$. Then $\{a,b\}\adjacent\{a',b'\}$.
\end{proposition}

\begin{proof}[Proof of Proposition~\ref{prop:SemiGermeAdjacent}]
  Suppose first that $a$ and $a'$ coincide along some arc with $x$ as 
  an end-point, and are
  disjoint apart from this arc. By observation~\ref{obs:CCarc}, there exists
  a curve $d$ passing through $x$ such that $T(a,d)$ and $T(a',d)$.
  By Lemma~\ref{lem:SpheresConnexes}, this implies $(a,d)\adjacent(a',d)$,
  and by the same lemma we also have $(a,b)\adjacent(a,d)$ and
  $(a',b')\adjacent(a',d)$. Hence $(a,b)\adjacent(a',b')$.
  
  Now we do not make the assumption any more that $a$ and $a'$ meet
  only along an arc. Still, thanks to the hypothesis of the proposition, 
  we may choose a set $V$ homeomorphic to a closed disk, with $x$ on its 
  boundary, and such that $a\cap V=a'\cap V$ is an arc whose endpoints are $x$ 
  and some other point $y$.
  Lemma~\ref{lem:ArcConnexeBis}, applied to the surface
  $\Sigma\smallsetminus \ring V$, 
  provides a
  sequence $a_0=a$, \ldots, $a_n=a'$, of nonseparating curves such that
  for all $i$, the curves $a_i$ and $a_{i+1}$ intersect only along
  the arc $a\cap V$, hence we may conclude by applying iteratively
  the reasoning above.
\end{proof}

\section{Local subgraphs}\label{sec:LocalSubgraphs}

In section~\ref{sec:Adjacency}, we considered
edges
$\{a,b\}, \{a',b'\}$ in the graph $\TransFin(\Sigma)$
satisfying $|a\cap b|=|a'\cap b'|=1$, and $\point(a,b) = \point(a',b')$.
We defined and used the equivalence relation $\adjacent$.
The aim of this section is to provide a geometric interpretation of the
equivalence classes. The results here are not used anywhere else in the paper.
In particular, this section is not used in the proof of our main results.
Nevertheless, we think it may help the reader to get a clear picture of
the situation.

\subsection{The graph of germs}
Let $x$ be a marked point in the surface $\Sigma$. In this section the we will study the local geometry of curves near $x$, so we may assume that $(\Sigma,x) = (\R^2, 0)$ whenever this is convenient. Given two simple arcs $a, a': [0,1] \to \Sigma$ with $a(0)=a'(0)=x$, we say that $a$ and $a'$ \emph{locally coincide} at $x$ if there exists a neighborhood $V$ of $x$ such that $a([0,1]) \cap V = a'([0,1]) \cap V$. This is an equivalence relation, whose equivalence classes are called \emph{germs of simple arcs at $x$}. The germ of $a$ is denoted $[a]_x$.
We say that $a$ and $a'$ \emph{locally intersect only at $x$} if there exists a neighborhood $V$ of $x$ such that 
$a([0,1]) \cap a'([0,1]) \cap V = \{x\}$. This second relation obviously induces a relation on germs.
Let us consider the graph $\cA(x)$ whose vertices are the germs of simple arcs at $x$, with an edge between the germs of $a$ and $a'$ whenever $a$ and $a'$ locally intersect only at $x$. 

This graph is \emph{not} connected, in fact it has infinitely (uncountably) many connected components, as we will see below. We postpone the description of the connected components to explain the relation with the adjacency relation defined in section~\ref{sec:Adjacency}.
We say that two vertices $\alpha, \alpha'$ of the graph $\cA(x)$ are \emph{comparable} if they belong to the same connected component of the graph.

\subsection{Germs and adjacency}

Given a point $x$ in $\Sigma$ and a simple closed curve $a$ in $\Sigma$ that contains $x$, we choose any one of the two germs of simple arc at $x$ included in $a$ and denote it by $\lfloor a \rfloor_{x}$. Which one of the two germs is chosen will not matter in what follows.

\begin{proposition}
\label{pro:AdjacencyGerms}
Let $a, b, a', b'$ be vertices in  $\TransFin(\Sigma)$ such that $T(a,b) $ and $T(a', b')$ hold, and assume 
$\Point(a,b) = \Point(a',b')$. Then 
$\{a,b\} \adjacent \{a',b'\}$  if and only if the germs $\lfloor a\rfloor_x$ and $\lfloor a'\rfloor_x$ are comparable.
\end{proposition}

Proposition~\ref{pro:AdjacencyGerms} will be proved in section~\ref{sec:ProofAdjacencyGerms} below. The aim of the next two sections is to provide a simple characterization of distance, and connected components, in the graph of germs; see Proposition~\ref{Prop:DistanceAndWidth} below.

\subsection{Distance in local subgraphs}

In this section, we give a geometrical interpretation of the distance in three different graphs, which are very much like the graph of germs $\cA(x)$.

Let $\Sigma$ be one of the following surfaces: (1) the compact annulus $\bbS^1 \times [0,1]$, (2) the open annulus $\bbS^1 \times \R$, or (3) the 2-torus $\T^2 =\bbS^1 \times \bbS^1$. 
We consider non-oriented simple arcs in $\Sigma$, more precisely simple curves connecting both sides of the annulus in case (1), properly embedded images of the real line connecting both ends of the open annulus in case (2), or simple closed curves in a fixed homotopy class, say homotopic to $\{0\}\times \bbS^1$ in case (3).
Let $\cA$ denote the graph whose vertices are one of the three above family of curves, with an edge between two curves whenever they are disjoint.

In order to express geometrically the distance in $\cA$, let us consider the cyclic cover $p : \widetilde \Sigma \to \Sigma$, respectively in case (1), (2), (3)
$$
p:\R \times [0,1] \to \R /\Z \times [0,1],
\ \ \ \ 
p:\R \times \R \to \R /\Z \times \R,
\ \ \ \ 
p:\R \times \bbS^1 \to \R /\Z \times \bbS^1
$$
given by the formula $p(x,y) = (x \text{ mod } 1, y)$.
Let $T$ be the deck transformation $(x, y) \to (x+1, y)$.

Now consider two curves $a,b$ which are vertices of the graph $\cA$.
Let $\widetilde a, \widetilde b$ be respective lifts of $a,b$ under the covering map $p$.
Note that the set 
$$
\{k \in \Z, \ \ T^k(\widetilde{a}) \cap \widetilde{b} \neq \emptyset \}
$$ 
is an interval of $\Z$, which is finite in the compact cases (1) and (3) but may be infinite in the open annulus case (2).
We define the \emph{relative width} $\width(a,b)$ as the cardinal of this set. This is an element of $\{0, 1, \dots, +\infty \}$.
The reader may check easily that $\width(a,b) = \width(b,a)$.

\begin{proposition}
  For every vertices $a \neq b$ of the graph $\cA$, the distance in the graph
  is given by
  $$
  d(a,b) = \width(a,b) + 1.
  $$
  In cases (1) and (3), the graph $\cA$ is connected.
  In case (2), $a$ and $b$ are in the same connected component of $\cA$ if
  and only if $\width(a,b) < +\infty$. 
\end{proposition}

\begin{proof}	
Let $a,b$ be as in the statement, and denote $w = \width(a, b)$. We first assume that $w < +\infty$.
By Schoenflies' theorem (in case (2), applied in the two-point compactification of the annulus, which is a sphere),
we may assume that $a$ is a vertical curve whenever this makes our life easier.
We first note that if $w=0$ then $a$ and $b$ admit lifts that are disjoint from every $T$-translate of each other, which shows that $a$ and $b$ are disjoint, and thus $d(a,b)=1$.
Let us now assume $w>0$, and prove the two following key properties.

\begin{enumerate}
\item[(i)] For every vertex $a'$ of $\cA$ such that $d(a,a')=1$, 
$$\width(a',b) \geq \omega-1.$$
\item[(ii)] There exists a vertex $a'$ of $\cA$ such that $d(a,a')=1$ and
$$\width(a',b) \leq w-1.$$
\end{enumerate}

To prove the first property, consider $a'$ such that $d(a,a')=1$.
By definition of the width $w$, we may find lifts $\widetilde{a}, \widetilde{b}$ of $a,b$ such that $\widetilde b$ is disjoint from $\widetilde{a}, T^{w+1}\widetilde{a}$ but meets 
$T(\widetilde{a}), \dots, T^w(\widetilde{a})$. Since $a$ and $a'$ are disjoint, there is a lift $\widetilde{a'}$ of $a'$ which is between $\widetilde{a}$ and $T(\widetilde{a})$. Then the curves
$$
T(\widetilde{a'}), \dots, T^{w-1}(\widetilde{a'})
$$
are between the two curves $T(\widetilde{a})$ and $T^w(\widetilde{a})$, and those two curves are not in the same connected component of $\widetilde{\Sigma} \setminus T^i(\widetilde{a'})$, for $i=1, \dots, w-1$. Since the curve $\widetilde b$ is connected and meets the two curves 
$T(\widetilde{a})$ and $T^w(\widetilde{a})$, it must meet all the $T^i(\widetilde{a'})$.
This proves that $\width(a', b) \geq w-1$.

Let us prove the second property. We consider $\widetilde{a}, \widetilde{b}$ as above. Let $S$ denote the compact strip or annulus bounded by $\widetilde a \cup T(\widetilde{a})$.
Remember that $\widetilde{b}$ meets $T(\widetilde{a})$ but not $\widetilde{a}$. Thus $\widetilde{b} \cap S$ is included in a (maybe infinite) family of \emph{bigons}, \textsl{i.e.}, topological disks bounded by a simple closed curve made of a segment of the curve $T(\widetilde{a})$ and a segment of the curve $\widetilde{b}$. Let $S^+$ denote the union of these bigons.
Symmetrically, the curve $\widetilde{b'} := T^{-w}(\widetilde{b})$ meets $\widetilde{a}$ but not $T(\widetilde a)$. Thus $T^{-w}(\widetilde{b}) \cap S$ is included in a union $S^-$ of bigons formed by the curves $\widetilde{a}$ and $\widetilde{b'}$.
A key point is that the sets $S^-$ and $S^+$ are disjoint, because the curves $\widetilde b$ and $\widetilde b'$ are disjoint, since $b$ is simple. Thus we may construct a homeomorphism $H$ supported in $S$ such that $H(S^-)$ is included in an arbitrarily small neighborhood of $\widetilde a$, and $H(S^+)$ is included in an arbitrarily small neighborhood of $T(\widetilde{a})$. In particular, we may find a curve $\widetilde{a'}$, which is a lift of some element $a'$ of $\cA$, included in the interior of $S$ and disjoint from both $S^-$ and $S^+$ (to be more explicit, take $\widetilde a' = H^{-1}(\{1/2\} \times [0,1])$ in the annulus case, in coordinates for which $a$ is the vertical curve $\{0\} \times [0,1]$). Note that $\widetilde{a'}$ is disjoint from $\widetilde{b}$ and $T^{-w}(\widetilde{b})$, and separate both curves, \textsl{i.e.}, the first one is on the right-hand side of $\widetilde{a'}$, and the second one is on the left-hand side. Thus the set
$$
\{k \in \Z, \ \ T^k(\widetilde{b}) \cap \widetilde{a'} \neq \emptyset \}
$$ 
has cardinality at most $w-1$. Which proves that $\width(a', b) \leq w-1$, as wanted.


Using (i) and (ii), an induction on $n$ shows that $d(a,b)=n$ if and only if
$\width(a,b)+1=n$, which completes the proof in the case when $\width(a,b)$ is finite. When $\width(a,b) = +\infty$,
 an argument analogous to property (i) above shows that $\width(a',b)=+\infty$ for every $a'$ such that $d(a,a')=1$. This shows that $a$ and $b$ are not in the same connected component of the graph. This completes the proof of the proposition.
\end{proof}

\subsection{Distance in the graph of germs}

Let us go back to the graph of germs $\cA(x)$. Assume $(\Sigma,x) = (\R^2, 0)$. Given two vertices $a,b$ of $\cA(x)$, we define their \emph{local relative width} $\width(a,b)$ as follows. The plane minus the origin is identified with the open annulus $\bbS^1 \times \R$, and we consider the graph $\cA$ from the previous section in the open annulus case. Then $\width(a,b)$ is defined as the infimum of the quantity $\width(A,B)$, where $A$ and $B$ are vertices of $\cA$ whose germs respectively equal $a$ and $b$. Here is a more practical definition, which is easily seen to be equivalent. Consider the universal cover $p: \widetilde \Sigma \to \Sigma$ as above. Abuse the definition by still denoting $a,b:[0,1] \to \Sigma$ two curves with $a(0)=b(0)=0$ whose germs respectively equal $a,b$. Let $\widetilde a, \widetilde b$ denote lifts of (the restrictions to $(0,1]$ of) $a,b$ in $\widetilde \Sigma$.
Then the number $\width(a,b)=w$ is characterized by the two following properties:

\begin{itemize}
\item[(i)] for every $t_0 \in (0,1]$, the restriction of $\widetilde a$ to $(0, t_0]$ meets at least $w$ integer translates of $\widetilde b$;

\item[(ii)] there exists $t_0 \in (0,1]$ such that the restriction of $\widetilde a$ to $(0, t_0]$ meets exactly $w$ integer translates of $\widetilde b$;
\end{itemize}

Analogously to the previous section, the distance in the  graph of germs is characterized by the local relative width.
\begin{proposition}\label{Prop:DistanceAndWidth}
Let $a \neq b$ be two vertices of the graph $\cA(x)$. Then $a$ and $b$ are in the same connected component of $\cA(x)$ if and only if $\width(a,b) < +\infty$. 
In this case, the distance in the graph is given by
$$
d(a,b) = \width(a,b) + 1.
$$
\end{proposition}

The proof is very similar to the proof in the previous section. Details are left to the reader.

\subsection{Proof of Proposition~\ref{pro:AdjacencyGerms}}
\label{sec:ProofAdjacencyGerms}

Let $a, b, a', b'$ be vertices of $\TransFin(\Sigma)$ such that $T(a,b)$ and $T(a', b')$ hold, and assume
$\Point(a,b) = \Point(a',b')$. Denote $x$ the common intersection point.

If $c$ is another vertex such that  $\{a, b, c\}$ is a 3-clique of type bouquet or extra bouquet, then the germs $\lfloor a\rfloor_x$ and $\lfloor c\rfloor_x$ are disjoint, thus obviously comparable. This entails the direct implication in Proposition~\ref{pro:AdjacencyGerms}.

Let us prove the converse implication. We assume that the germs $\lfloor a\rfloor_x$ and $\lfloor a'\rfloor_x$ are comparable. In other words, there exists arcs $\alpha_0, \dots, \alpha_n$ with $\alpha_i(0)=x$ and whose sequence of corresponding germs is a path from 
$\lfloor a\rfloor_x$ to $\lfloor a'\rfloor_x$ in the graph of germs.
Note that each germ $\alpha_i$ may be extended to a non separating curve $a_i$, and we can find another non separating curve $b_i$ such that $T(a_i, b_i)$ holds. Thus the end of the proof is a direct consequence of the following lemma.

%
%
%

\begin{lemma}
Let $a, b, a', b'$  be vertices in  $\TransFin(\Sigma)$ such that $T(a,b)$ and $T(a', b')$ hold. Assume that for some choices $\lfloor a\rfloor_x, \lfloor a'\rfloor_x$ of arcs at $x$ included respectively in $a$ and $a'$, the germs $\lfloor a\rfloor_x, \lfloor a'\rfloor_x$ intersect only at $x$. Then  $\{a, b\} \adjacent \{a', b'\}$.
\end{lemma}

\begin{proof}[Proof of the lemma]
Let $c$ be an arc that contains $x$ in its interior and locally coincides with 
$\lfloor a\rfloor_x \cup \lfloor a'\rfloor_x$. Extend $c$ into a non separating closed curve, still denoted $c$, and consider any other non separating curve $d$ such that $T(c,d)$ holds. Since $c$ locally ``half coincides'' near $x$ with both $a$ and $a'$, we may apply Proposition~\ref{prop:SemiGermeAdjacent} twice, and get that 
$\{a,b\} \adjacent \{c,d\} \adjacent \{a', b'\}$.
\end{proof}

\subsection{Curves and diffeomorphisms}\label{ssec:Tourbillons}
In this short subsection we explain how one can use the fine curve graph to detect fundamental non differentiability.

Let $\Phi$ be an automorphism of $\TransFin(\Sigma)$.
We introduce the following property $D(\Phi)$: 

\emph{For every vertices  $a$, $b$ of $\TransFin(\Sigma)$ such that $T(a,b)$ holds, if $\point(\Phi(a),\Phi(b)) = \point(a,b)$ then there exists $a', b'$ such that $T(a', b')$ holds, $\point(a',b') = \point(a,b)$ and $\{\Phi(a'), \Phi(b') \} \adjacent \{a', b'\}$.}

 Note that this property is clearly invariant under conjugacy in the group of automorphisms. Let $h$ be a homemorphism of $\Sigma$, and denote $\Phi = \Phi_h$ the action of $h$ on the graph $\TransFin(\Sigma)$.
\begin{observation}
If $h$ is differentiable everywhere, then property $D(\Phi_h)$ holds.
\end{observation}

Indeed, hypothesis $\point(\Phi(a),\Phi(b)) = \point(a,b)$ is equivalent to the fact that the point $x = \point(a,b)$ is a fixed point of $h$. Since $h$ is differentiable at $x$, it is easy to check that every germ of smooth arc at $x$ is comparable to its image. Take any two smooth curves $a', b'$ such that $T(a', b')$ and $\point (a', b') = \point(a,b)$, then Proposition~\ref{pro:AdjacencyGerms} tells us that $\{\Phi(a'), \Phi(b') \} \adjacent \{a', b'\}$.

Now consider a particular homeomorphism $h$ of $\Sigma$ and assume that $h$ admits a fixed point where, for some local polar coordinates, $h$ writes
$$
(r, \theta) \mapsto (r, \theta + \frac{1}{r}).
$$

\begin{observation}
Property $D(\Phi_h)$ does not hold.
\end{observation}

An easy proof of this is obtained by considering the
\emph{local rotation interval} of $h$ at $x$, as defined in~\cite{FredLocRot},
section 2.3.
Indeed, the local rotation interval of $h$ at $x$ equals $\{+\infty\}$,
which accounts for the fact that orbits turn faster and faster around $x$, in
the positive direction, as we get nearer and nearer to $x$ (the quickest way to
check this is to show that the \emph{local rotation set} of $h$ at $x$
is $\{+\infty\}$, and then to apply Théorème 3.9 of~\cite{FredLocRot} that
relates the local rotation set and the local rotation interval). We argue by
contradiction to show that property $D(\Phi_h)$ does not hold.
Assuming property $D(\Phi_h)$ holds, consider curves $a,b$ such that $T(a,b)$
holds and $\point(a, b) = x$. Let $a', b'$ be given by property $D(\Phi_h)$,
such that $\{\Phi(a'), \Phi(b') \} \adjacent \{a', b'\}$.
The reverse direction of Proposition~\ref{pro:AdjacencyGerms} tells us that the
germs of $h(a')$ and $a'$ are comparable at $x$. This entails easily, from the
definition, that the local rotation interval of $h$ at $s$ is a bounded
interval, a contradiction.

\section{Fine graph of smooth curves}\label{sec:AutCFinLisse}

In this section we address the case of smooth curves, and
prove Theorem~\ref{thm:AutCFinLisse} and Proposition~\ref{prop:PasDiff}.
In all the section, $\Sigma$ will be a connected, nonspherical
surface without boundary, endowed with a smooth structure.
In Section~\ref{sec:configSmoothCurves} we will restrict to the orientable case.

\subsection{From bijections to higher regularity}

One step in the proof of Theorem~\ref{thm:AutC1} was
Proposition~\ref{prop:BijHomeo}, in which we proved
that if an automorphism of $\TransFin(\Sigma)$ is
supported by a bijection of $\Sigma$, then that
bijection is a homeomorphism of $\Sigma$.

We may ask the same question about automorphisms of
$\TransFinLisse(\Sigma)$, and this paragraph is devoted
to the proof of the following two statements.
We denote by $\Homeo_{\infty \pitchfork}(\Sigma)$ the group of
bijections of $\Sigma$ which preserve the family of
smooth, nonseparating closed curves, and preserve transversality
between such curves.
The first statment below justifies this notation. Here, for simplicity we
restrict to the case of orientable surfaces.

\begin{proposition}\label{prop:BijLisseHomeo}
  Let $\Sigma$ be a connected, non spherical orientable surface.
  The group  $\Homeo_{\infty \pitchfork}(\Sigma)$ is contained in
   $\Homeo(\Sigma)$.
\end{proposition}
%

\begin{proof}[Proof of Proposition~\ref{prop:BijLisseHomeo}]
%
  Let $h \in  \Homeo_{\infty \pitchfork}(\Sigma)$. We will prove that the image under $h$
  of any open set is an open set. This is the continuity of $h^{-1}$, and by
  applying the argument to $h$ we also get the continuity of $h$.
  
  To do this, we only need to consider the images of a family of sets that
  generates the topology. Given three non separating curves $a,b,c$, we denote
  $V(a; b, c)$ the union of all the non separating curves $d$ that meet $a$
  and are disjoint from $b$ and $c$.
  
  \begin{observation}
    The set $V(a ; b,c)$ is the union of some of the connected components of the
    complement of $b \cup c$ that meet $a$. In particular, it is an open set.
  \end{observation}
  Indeed, let $x$ be a point of $V(a; b,c)$. By definition there is a non
  separating curve  $d$ passing through $x$ and meeting $a$ but not $b$ nor $c$.
  Consider another point $y$ that belongs to the connected component $V_x$ of
  the complement of $b \cup c$ that contains $x$. By modifying $d$ using an arc
  connecting $x$ to $y$ in $V_c$, we find another curve $d'$, isotopic to $d$,
  still meeting $a$ but not $b$ nor $c$, and passing through $y$. This proves
  that $V(a ; b,c)$ contains $V_x$, and the observation follows.
  
  Now let $a$ be a nonseparating curve.
  Let $a^+, a^-$ be obtained by pushing $a$ to both sides. Then $V(a; a^+, a^-)$
  is a neighborhood of $a$, and by making $a^+$ and $a^-$ vary we get a basis
  of neighborhoods ${\mathcal B}(a)$ of the curve $a$.
  The union of all these families ${\mathcal B}(a)$ clearly 
  generates the topology of~$\Sigma$.
  
  Thus it suffices to check that the image under $h$ of each set $V(a; b, c)$
  is an open set. But since $h$ is a bijection, we have
  \[ h(V(a; b, c)) = V( h(a); h(b), h(c)). \]
  By hypothesis $h(a), h(b), h(c)$ are nonseparating closed curves, and by the
  observation this set is open.
\end{proof}

We now prove Proposition~\ref{prop:PasDiff} stated in the introduction,
namely the existence of elements of $\Homeo_{\infty \pitchfork}(\Sigma)$ that are not smooth.

\begin{proof}[Proof of Proposition~\ref{prop:PasDiff}]
  We will construct a homeomorphism
  $F\colon\R^2\to\R^2$, which is not differentiable at the origin,
  but such that both $F$ and $F^{-1}$ send smooth curves to
  smooth curves. The construction can easily be modified to make $F$ compactly    
  supported, and then be transported on our surface~$\Sigma$. It will be
  clear from the constuction that this map preserves transversality.
  
  Let $h\colon\R\to\R$ be a smooth diffeomorphism supported in the
  segment $[1/2, 2]$. That is to say, $h(x)=x$ for all $x$ outside
  $[1/2,2]$; we suppose however that $h(1)\neq 1$.
  We consider the map $F$ defined by $F(x,y)=(x, xh(y/x))$ if
  $x\neq 0$, and $F(x,y)=(x,y)$ otherwise. We claim that this map
  has the desired property.
  
  This map, as well as its inverse, is obviouly smooth in
  restriction to $\R^2\smallsetminus\{(0,0)\}$. Direct computation
  shows that $F$ has directional derivatives in all directions
  around the origin,
  but the ``differential''  fails to be linear:
  both partial derivatives are
  those of the identity, while the
  directional derivative in the direction $(1,1)$ is not.
  So $F$ is not differentiable at the origin.
  
  Now, let $\gamma\colon\R\to\R^2$ be a smooth, proper embedding.
  If $(0,0)$ is not in the image of $\gamma$, then of course,
  $F\circ\gamma$ is still smooth. So suppose, say, that
  $\gamma(0)=(0,0)$.
  If $\gamma'(0)$ lies outside the two (opposite) sectors of
  vectors of slopes between $1/2$ and $2$, then $F\circ\gamma$
  and $\gamma$ have the same germ at $0$.
  Otherwise,
  and up to reparameterization,
  we can write, near $0$,
  $\gamma(t) = (t, \alpha(t))$ where $\alpha$ is a smooth map
  (satisfying $\alpha(0)=0)$),
  from a neighborhood of $0$, to~$\R$. This yields the formula
  \[ F\circ\gamma(t) = \left(t, th \left(\frac{\alpha(t)}{t}\right)\right).  \]
  Now, the smoothness of $F\circ\gamma$
  follows from the following elementary observation.
  
  \noindent
  {\bf Claim.}
  {\em Let $\alpha\colon\R\to\R$ be a smooth map
  satisfying $\alpha(0)=0$. Then the map $t\mapsto \frac{\alpha(t)}{t}$
  when $t\neq 0$, and $\alpha'(0)$ when $t=0$, is smooth.}
  
  Indeed, by the fundamental theorem of Calculus, for all $t\in\R^\ast$ we have
  \[ \frac{\alpha(t)}{t} = \int_0^1 \alpha'(ts)ds, \]
  and this integral with parameter can be differentiated
  indefinitely.\footnote{We borrow this elegant argument from~\cite{stack}.}
  %
\end{proof}


\subsection{A weak convergence for sequences of curves}

In order to prove~Theorem~\ref{thm:AutCFinLisse}, we now explain how to recognize configurations of smooth curves.
Given two vertices $a$ and $b$ of the graph $\TransFinLisse(\Sigma)$,
we will denote by $a-b$ the property that they are neighbors in the graph.
If $(f_n)_{n\in\N}$ is a sequence of vertices, we denote by
$(f_n)_{n\in\N}-a$ the property that for all $n$ large enough, $f_n-a$.
The first property of sequences we may recover from the graph is the
distinction of what curves go to infinity.
\begin{lemma}\label{lem:fnSenVa}
  Let $(f_n)_{n\in\N}$ be a sequence of vertices of $\TransFinLisse(\Sigma)$.
  The following are equivalent.
  \begin{itemize}
  \item[-] for all $d$, we have $(f_n)_{n\in\N}-d$;
  \item[-] for every compact subset $K$ of $\Sigma$, for every $n$ large
    enough, $K\cap f_n=\emptyset$.
  \end{itemize}
\end{lemma}
\begin{proof}
  The second statement obviously implies the first, as we may just
  take $K=d$. Let us prove the converse implication by contraposition.
  Suppose $K$ intersects infinitely many $f_n$. Since $K$ is compact,
  there is a point $x\in K$, such that every neighborhood of $x$
  intersects infinitely many $f_n$. We consider two open sets
  $B_1$ and $B_2$ with $x\in B_1$ and $\overline{B_1}\subset B_2$,
  and three bottle-shaped arcs, as in Figure~\ref{fig:3Bouteilles}.
  These arcs may be continued to form three nonseparating closed curves,
  $d_1$, $d_2$ and $d_3$. Now, let $n$ be such that $f_n$ enters $B_1$.
  If $f_n$ does not enter nor leave $B_2$ through the neck of the bottle
  corresponding to $d_1$, then we cannot have $f_n-d_1$, since $f_n$ and
  $d_1$ have to intersect at least twice. Hence, $f_n$ passes
  through the neck of $d_1$, and in order to impose that $f_n-d_2$, another
  arc of $f_n$ has to get out of the bottle corresponding of $d_2$
  through its neck. But then $f_n$ has to meet $d_3$ twice, and we cannot have 
  $f_n-d_3$. In other words, for
  all $n$ such that $f_n$ enters $B_1$, we can't have $f_n-d_1$ and $f_n-d_2$
  and $f_n-d_3$, hence the first statement is not true, and our
  implication is proved.
  \begin{figure}[htb]
  \begin{asy}
    import geometry;
    
    fill(circle((0,0),35), lightgray);
    
    fill(circle((0,0),17), lightgreen);
    label("{\small $B_1$}", (0,0), NW);
    label("{\small $B_2$}", 35*dir(0), E);
    
    dot((0,0)); label("{\small $x$}", (0,0), SE);
    
    real r1 = 29, r2 = 25, r3 = 21;
    path d1 = (-40, 8)--(-32,8){right}..(r1*dir(135)){dir(45)}..(r1*dir(90)){right}..
      (r1*dir(0)){down}..(r1*dir(-90)){left}..(r1*dir(-135)){dir(135)}..
      {left}(-32,-8)--(-40,-8);
    
    path d2 = (-40, 8)--(-32,8){right}..(r2*dir(135)){dir(45)}..(r2*dir(90)){right}..
      (r2*dir(0)){down}..(r2*dir(-90)){left}..(r2*dir(-135)){dir(135)}..
      {left}(-32,-8)--(-40,-8);
    path d2tourne = rotate(120)*d2;
    
    path d3 = (-40, 8)--(-32,8){right}..(r3*dir(135)){dir(45)}..(r3*dir(90)){right}..
      (r3*dir(0)){down}..(r3*dir(-90)){left}..(r3*dir(-135)){dir(135)}..
      {left}(-32,-8)--(-40,-8);
    path d3tourne = rotate(-120)*d3;
    
    draw(d1, blue);  draw(rotate(120)*d2, red);  draw(rotate(-120)*d3, green);
    
    label("{\small Any arc from $B_1$ to the outside of $B_2$ and}", (0,-50));
    label("{\small disjoint from $d_1$ must meet $d_2$ and $d_3$.}", (0,-64));
  \end{asy}
  \caption{Three bottles}
  \label{fig:3Bouteilles}
  \end{figure}
\end{proof}
Thus, we will say here that a sequence $(f_n)_{n\in\N}$ is {\em relevant}
if it has no subsequence $(f_{\varphi(n)})_{n\in\N}$ such that for
all $d$, $(f_{\varphi(n)})_{n\in\N}-d$.
We now explore, for such sequences, the following notion of convergence.
We say that a relevant sequence $(f_n)$ {\em converges in a weak sense}
to a curve $a$ if for every $d$ such that $a-d$, we have $(f_n)-d$.
We denote this property by~$W((f_n),a)$.
\begin{lemma}\label{lem:ConvFaible}
  Let $(f_n)_{n\in\N}$ be a sequence of vertices, and $a$ be a vertex
  of $\TransFinLisse(\Sigma)$.
  \begin{itemize}
  \item[$\bullet$]
  If $W((f_n),a)$, then the sequence $(f_n)_{n\in\N}$ converges
  in the Hausdorff topology to $a$: for every neighborhood $V$ of the
  curve $a$, for all $n$ large enough, we have $f_n\subset V$.
  \item[$\bullet$]
  If the sequence $(f_n)_{n\in\N}$ of curves, with some appropriate
  parameterization, converges in $C^1$-topology to $a$, then
  $W((f_n)_{n\in\N},a)$.
  \end{itemize}
\end{lemma}
\begin{proof}
  We first prove the first point. Let us first mention that the statement
  we wrote is indeed equivalent to the Hausdorff convergence, because
  any essential curve in a small enough neighbourhood of $a$ must pass close
  to every point of $a$, and thus is Hausdorff-close to $a$.
  Suppose for contradiction that
  for some neighborhood $V$ of $a$, we have $f_n\not\subset V$
  infinitely often. Then we may find a point $x$, not in $\overline{V}$,
  such that every neighborhood of $x$ is visited by infinitely many
  $f_n$. We may construct three bottle-shaped arcs around $x$ exactly
  as in the proof of the preceding lemma, and complete these arcs
  to non separating simple closed curves $d_1$, $d_2$, $d_3$, which
  can be requested to satisfy $a-d_i$ for $i=1,2,3$. Then we cannot
  have $f_n-d_i$ for all $i\in\{1,2,3\}$ for the same reason as in
  this preceding proof, and this contradicts the hypothesis
  that~$W((f_n),a)$ holds.
  
  The second point is the well known stability of transversality in
  the $C^1$-topology.
  %
\end{proof}
\begin{remark}
  In fact, the condition $W((f_n), a)$ implies $C^0$-convergence, in the
  following sense. Given a parameterization $\alpha$ of $a$, we can choose the
  parameterizations of the $f_n$'s yielding a sequence of parmaetrized curves 
  converging uniformly to~$\alpha$.
  
  As we will not use this fact, we only sketch a quick argument. Let $V$ be
  a small tubular neighborhood of $a$, and let $d_1,\ldots,d_N$ be simple
  closed nonseparating curves, each meeting $a$ transversely at one point,
  and cutting $V$ in small chunks $V_1,\ldots,V_N$ that are met by $\alpha$ in 
  that
  cyclic order. Let $n$ be large enough so that $f_n\subset V$ and
  $f_n-d_i$ for each $i$. Then $f_n\cap V_i$ is connected, and $f_n$
  visits the pieces $V_1,\ldots,V_N$ in that order. Hence we may choose
  a parameterization of $f_n$, say, $F_n$, in such a way that for all
  $i\in\{1,\ldots,N\}$ and all $t$, $F_n(t)\in V_i$ if and only if
  $\alpha(t)\in V_i$. This implies that $F_n$ is uniformly close to $\alpha$.
\end{remark}
This notion is actually somewhere strictly in between
$C^0$-convergence and $C^1$-convergence,
as we remark in the following
example. This construction will play a crucial role below in the proof of
Theorem~\ref{thm:AutCFinLisse}.
\begin{example}\label{eg:ConvFaible}
  Let $a$ be a smooth nonseparating curve in $\Sigma$. We choose
  a point $p$ in $a$, and a chart around one of its point, diffeomorphic
  to $\R^2$, in such a way that $a$ corresponds to the axis of
  equation $y=0$ in that plane. In this chart, we consider the
  functions
  $f_1\colon x\mapsto\frac{2}{1+x^2}$,
  and for all $n\geqslant 2$,
  $f_n\colon x\mapsto\frac{f_1(nx)}{n}$.
  Abusively, we still denote their graphs by the same letters,
  and then, we may extend these arcs, viewed in $\Sigma$, to
  simple closed curves (consisting of pushing $a$ aside), that
  converges $C^1$ to $a$ outside of the point $p$.
  Abusively we still use the same letters $f_n$ to denote these
  closed curves.
  
  Obviously, the sequence $(f_n)_{n\geqslant 1}$ does not
  converge $C^1$ to $a$, because $f_n$ has slope $-1$ at the
  point $(1/n, 1/n)$.
  
  Nonetheless, we claim that $W((f_n),a)$ holds.
  Indeed, let $d$ be such that $a-d$. Since the sequence $(f_n)$
  converges $C^1$ to $a$ everywhere except at the origin of this
  $\R^2$ chart, the only case in which it is not already clear
  that $(f_n)-d$ is when $d$ meets $a$ transversely at the origin.
  If $d$ has a strictly positive slope there, then for $n$ large
  enough, the intersection $f_n\cap d$ will be transverse because
  the slopes of $f_n$ are all negative in the region $x>0$.
  The case when $d$ has negative slope is symmetric, and if $d$
  has vertical slope, it will be transverse with $f_n$ since these
  have bounded slopes.
\end{example}

\subsection{Recognizing configurations of smooth curves}
\label{sec:configSmoothCurves}
In this last section we assume that our surface $\Sigma$ is orientable.

If $a$, $b$ are vertices of $\TransFinLisse(\Sigma)$, we will
denote by $D_\infty(a,b)$ the condition that $a-b$ and
for all sequences $(f_n)$ and $(g_m)$ such that
$W((f_n),a)$ and $W((g_m),b)$, we have $f_n-g_m$ for all
$m, n$ large enough.

\begin{lemma}\label{lem:DisjointLisse}
  Let $a$, $b$ be smooth nonseparating curves. Then
  $D_\infty(a,b)$ holds if and only if $a$ and $b$ are disjoint.
\end{lemma}
\begin{proof}
  Suppose first that $a$ and $b$ are disjoint. Then they admit
  disjoint neighborhoods, $V_1$ and $V_2$. For any sequences
  $(f_n)$ and $(g_m)$ with $W((f_n),a)$ and $W((g_m),b)$,
  for all $m,n$ large enough we have $f_n\subset V_1$ and
  $g_m\subset V_2$, by Lemma~\ref{lem:ConvFaible}. Hence,
  $f_n-g_m$ for all $m,n$ large enough, and $D_\infty(a,b)$
  holds indeed.
  
  Now, suppose that $a$ and $b$ are not disjoint. Since
  $a-b$, the curves $a$ and $b$ have a transverse intersection,
  and in an appropriate chart diffeomorphic to $\R^2$, the
  curves $a$ and $b$ correspond respectively to the axes
  $y=0$ and $x=0$.
  
  Then we may form a sequence $(f_n)$ such that $W((f_n),a)$
  exactly as in example~\ref{eg:ConvFaible}, and for
  $(g_n)$ we just exchange coordinates $x$ and $y$.
  For all $n$, the curves $f_n$ and $g_n$ have
  a non transverse intersection point (at $(1/n,1/n)$ in the
  chart of Example~\ref{eg:ConvFaible}), hence the condition
  $D_\infty(a,b)$ does not hold.
\end{proof}
In the end of the proof, the curves $f_n$ and $g_n$ were tangent
at their intersection point, hence not neighbors in the graph
$\TransFinLisse(\Sigma)$. This may look accidental, but
upon changing the formula of $f_1$ in Example~\ref{eg:ConvFaible}
to $x\mapsto \frac{3}{2+x^2}$, for example, we get three intersection
points.

\medskip

From now on, we restrict ourselves to the case of orientable
surfaces. One reason is that it would take more work to
recover the extra bouquets and not only the bouquets; one
other reason is that the next lemma works best when at least
one of $a$, $b$ or $c$ is two-sided.
\begin{lemma}\label{lem:BouquetLisse}
  Suppose $\Sigma$ is orientable.
  Let $\{a,b,c\}$ be a $3$-clique of $\TransFinLisse(\Sigma)$,
  and suppose that these three curves pairwise intersect.
  Then the following are equivalent.
  \begin{enumerate}
  \item This $3$-clique is of type bouquet.
  \item There exists a relevant sequence $(f_n)$ of vertices of
    $\TransFinLisse(\Sigma)$, which are all disjoint from $a$,
    and such that for all $d$ disjoint from $b$ and
    satisfying $c-d$, we have $(f_n)-d$. 
  \end{enumerate}
\end{lemma}
\begin{proof}
  Suppose $\{a,b,c\}$ is of type bouquet. Then the sequence
  $(f_n)_{n\in\N}$ can be constructed explicitly.
  Let $p=b\cap c$. Fix a (smooth) metric on $\Sigma$,
  we remove all points of the ball $B(p,\frac{1}{n})$ off the
  curves $b$ and $c$, this gives two arcs. There is a natural
  way of adding smooth subarcs of $B(p,\frac{1}{n})$ in order
  to extend this union of two arcs, to a curve $f_n$ which
  does not intersect $a$.
  In a one-holed torus neighborhood of $b\cup c$, with
  a choice of meridian and longitude coming from $b$ and $c$, these curves
  $f_n$ have slope $1$, or $-1$; these are indeed nonseparating simple
  closed curves.
  Now if $d$ is disjoint from $b$ and satisfies $c-d$, then
  either $d$ is disjoint from $c$, and then $f_n$ is disjoint from $d$
  for all $n$ large enough, or $d$ has a transverse intersection
  with $c$ at a point distinct from $p$, and we also have $f_n-d$ for all
  $n$ large enough. Thus, (1) implies~(2).
  
  Conversely, suppose~(2). We first claim that the sequence $(f_n)$ then
  concentrates into neighborhoods of $b\cup c$.
  For contradiction, suppose that we can find a neighborhood $V$ of
  $b\cup c$, such that $f_n\not\subset V$ for infinitely many $n$. Then, there
  exists a point $x$, with $x\not\in b\cup c$, and such that every
  neighborhood of $x$ meets infinitely many $f_n$. Then we may choose
  three bottle-shaped arcs around $x$, and complete them into curves
  $d_1$, $d_2$ and $d_3$ disjoint from $b$ and satisfying $d_j-c$ for
  $j=1,2,3$. Indeed, we may start with a curve $d_0$ obtained by
  pushing $b$ aside, and then perform surgeries on $d_0$.
  The same reasoning as in the proof of Lemma~\ref{lem:fnSenVa} shows
  that $f_n\not -d_j$
  for some $j\in\{1,2,3\}$ and for infinitely 
  many $n$,
  contradicting the hypothesis~(2).
  
  Now, suppose for contradiction that $\{a,b,c\}$ is a necklace. Then,
  for a sufficiently small regular neighborhood $V$ of $b\cup c$, we
  may observe that $V\smallsetminus a$ is contractible. Hence it cannot
  contain any nonseparating simple closed curve $f_n$, and the
  hypothesis~(2) cannot be fullfilled. This proves that (2) implies~(1).
\end{proof}

Now the proof of Theorem~\ref{thm:AutCFinLisse} is a straightforward
adaptation of the proof of Theorem~\ref{thm:AutC1}.
The statements about
connectedness of complexes of arcs, for example, are equivalent to
their counterparts with regularity, because of the argument of
homotopy recalled in the proof of
Lemma~\ref{lem:NonSepFinRelUnBordConnexe} and borrowed
from~\cite{BHW}.


\bibliographystyle{plain}

\bibliography{BibAutFin}

\end{document}